\documentclass[english,article]{siamart171218} 

\usepackage{etoolbox}
\providetoggle{comments}\settoggle{comments}{true}

\allowdisplaybreaks
\usepackage{amssymb,amsfonts,amsopn,mathrsfs,calligra}
\usepackage{graphicx,epstopdf,subfig}
\usepackage{algorithmic,array,multirow}
\usepackage{tikz,pgfplots}
\usetikzlibrary{patterns}
\usetikzlibrary{quotes,arrows.meta}
\usetikzlibrary{pgfplots.groupplots}
\usepackage{pgfplotstable}
\usepackage{hyperref}

\pgfplotsset{
    discard if/.style 2 args={
        x filter/.code={
            \edef\tempa{\thisrow{#1}}
            \edef\tempb{#2}
            \ifx\tempa\tempb
                
            \fi
        }
    },
    discard if not/.style 2 args={
        x filter/.code={
            \edef\tempa{\thisrow{#1}}
            \edef\tempb{#2}
            \ifx\tempa\tempb
            \else
                
            \fi
        }
    }
}

\definecolor{hcorange}{RGB}{245, 130, 48}
\definecolor{hcnavy}{RGB}{0, 0, 128}
\definecolor{hcblue}{RGB}{0, 130, 200}
\definecolor{hcpink}{RGB}{250, 190, 190}
\definecolor{hcyellow}{RGB}{255, 255, 255}
\definecolor{hcbrown}{RGB}{128, 0, 0}
\definecolor{hclavender}{RGB}{230, 190, 255}
\definecolor{hcgrey}{RGB}{128, 128, 128}
\definecolor{hcgreen}{RGB}{60, 180, 75}
\definecolor{hcred}{RGB}{230, 25, 75}

\makeatletter
\pgfplotsset{
    /pgfplots/flexible xticklabels from table/.code n args={3}{%
        \pgfplotstableread[#3]{#1}\coordinate@table
        \pgfplotstablegetcolumn{#2}\of{\coordinate@table}\to\pgfplots@xticklabels
        \let\pgfplots@xticklabel=\pgfplots@user@ticklabel@list@x
    }
}
\makeatother

\iftoggle{comments}
   {
    \usepackage{todonotes}
    \newcommand\ed[1]{\todo[color=black!10]{#1}}
    \newcommand\lc[1]{\todo[color=red!10]{#1}}
    \newcommand\eb[1]{\todo[color=blue!10]{#1}}
    
    \newcommand\edi[1]{\todo[color=black!10,inline]{#1}}
    \newcommand\lci[1]{\todo[color=red!10,inline]{#1}}
  }
  {\newcommand\ed[1]{} \newcommand\edi[1]{} \newcommand\lci[1]{} \newcommand\lc[1]{} \newcommand\ebi[1]{} \newcommand\eb[1]{}}

\DeclareFontFamily{U}{mathx}{\hyphenchar\font45}
\DeclareFontShape{U}{mathx}{m}{n}{
      <5> <6> <7> <8> <9> <10>
      <10.95> <12> <14.4> <17.28> <20.74> <24.88>
      mathx10
      }{}
\DeclareSymbolFont{mathx}{U}{mathx}{m}{n}
\DeclareFontSubstitution{U}{mathx}{m}{n}
\DeclareMathAccent{\widecheck}{0}{mathx}{"71}
\DeclareMathAccent{\wideparen}{0}{mathx}{"75}

\newcommand{\OO}[1]{\mathcal{O}\left(#1\right)}

\newcommand{\ske}{c}

\newcommand{\algo}{spaND}
\renewcommand{\H}{\mathcal{H}}
\newcommand{\st}{\text{size}_{\,\text{Top}}}
\newcommand{\mf}{\text{mem}_{\text{F}}}
\newcommand{\ncg}{\text{n}_{\text{CG}}}
\newcommand{\tf}{\text{t}_{\text{F}}}
\newcommand{\tp}{\text{t}_{\text{P}}}
\newcommand{\ts}{\text{t}_{\text{S}}}

\newcommand{\rev}[1]{#1}
\newcommand{\revrev}[1]{#1}

\DeclareGraphicsExtensions{.eps,.pdf,.png,.jpg}

\newcommand{\TheTitle}{An Algebraic Sparsified Nested Dissection Algorithm using Low-Rank Approximations}
\newcommand{\TheAuthors}{L. Cambier, C. Chen, E. Boman, S. Rajamanickan, R. Tuminaro, E. Darve}

\headers{Sparsified Nested Dissection}{\TheAuthors}

\title{{\TheTitle}}

\author{
    L{\'e}opold Cambier\thanks{Institute for Computational \& Mathematical Engineering, Stanford University, \email{lcambier@stanford.edu}}
    \and
    Chao Chen\thanks{Institute for Computational Engineering and Sciences, The University of Texas at Austin, \email{chenchao.nk@gmail.com}}
    \and
    Erik~G. Boman \thanks{Center for Computing Research, Sandia National Laboratories, \email{egboman@sandia.gov}}
    \and
    Sivasankaran Rajamanickam\thanks{Center for Computing Research, Sandia National Laboratories, \email{srajama@sandia.gov}}
    \and
    Raymond~S. Tuminaro\thanks{Center for Computing Research, Sandia National Laboratories, \email{rstumin@sandia.gov}}
    \and
    Eric Darve\thanks{Department of Mechanical Engineering, Stanford University, \email{darve@stanford.edu}}
}

\begin{document}

\maketitle

\begin{abstract} 
We propose a new algorithm for the fast solution of large, sparse, symmetric positive-definite linear systems,
spaND --- sparsified Nested Dissection.
It is based on nested dissection, sparsification and low-rank compression.
After eliminating all interiors at a given level of the elimination tree, the algorithm sparsifies all separators corresponding
to the interiors. 
This operation reduces the size of the separators by eliminating some degrees of freedom but without introducing
any fill-in. This is done at the expense of a small and controllable approximation error. 
The result is an approximate factorization that can be used as an efficient preconditioner.
We then perform several numerical experiments to evaluate this algorithm.
We demonstrate that a version using orthogonal factorization and block-diagonal scaling takes \rev{fewer} CG iterations
to converge than previous similar algorithms on various kinds of problems.
Furthermore, this algorithm is provably guaranteed to never break down and the matrix stays symmetric positive-definite throughout the process. 
We evaluate the algorithm on some large problems \rev{and} show it exhibits near-linear scaling. The factorization time is roughly
$\OO{N}$ and the number of iterations grows slowly with $N$.
\end{abstract}



\section{Introduction}

We are interested in solving large symmetric, positive-definite (SPD) sparse linear systems
\begin{equation}\label{eq:linear_system} Ax = b,\; A \in \mathbb{R}^{N \times N}. \end{equation}
In particular, we focus on linear systems with similar properties as those arising from the discretization of elliptic
partial differential equations, using finite difference or finite elements for instance. Solving such systems is a crucial part of many scientific simulations.

Algorithms for solving \autoref{eq:linear_system} are traditionally divided into three categories.
On one hand are direct methods. The naive Cholesky ($A = L L^\top$) factorization can lead to a factorization 
cost of $\OO{N^3}$ \rev{(with $\OO{N^2}$ memory use)} due to fill-in in the factor $L$.
When the matrix $A$ comes from the discretization of PDE's in 2D or 3D space, one usually uses the Nested Dissection 
\cite{lipton1979generalized} ordering to reduce fill-in. 
By doing so, the \rev{time} complexity is typically reduced to $\OO{N^{3/2}}$ (in 2D)
and $\OO{N^2}$ (in 3D), \rev{with the memory complexity reduced to $\OO{N \log N}$ (in 2D) and $\OO{N^{4/3}}$ (in 3D)}
\cite{george1973nested, lipton1979generalized}.
This is what most state-of-the-art direct solvers are built upon \cite{chen2008algorithm, amestoy2000mumps, henon2002pastix}. 
Those algorithms work very well for most moderate-size problems. However, the $\OO{N^2}$ complexity in 3D makes them 
intractable on large scale problems.

An alternative is to use iterative algorithms like Krylov methods or multigrid.
Multigrid \cite{fedorenko1962relaxation, bramble1993multigrid, hackbusch2013multi} (and its algebraic version, \cite{brandt1982algebraic, stuben2001review})
works very well on fairly regular elliptic \rev{PDEs}, usually with a near-constant iteration count \rev{and $\OO{N}$ memory use} regardless of the problem size. 
However, it can solve only a fairly limited range of problems and its iteration count can start growing when the problem becomes ill-conditioned.
Krylov methods, such as CG \cite{hestenes1952methods, van1992bi}, MINRES \cite{paige1975solution} or GMRES 
\cite{saad1986gmres} can be applied to a very wide range of problems, necessitating only sparse 
matrix-vector products. However, to converge at all, one needs to always couple them with an efficient preconditioner. 
This is typically a very problem dependent task.

One way, however, to build preconditioners is using incomplete factorizations and low-rank approximations. 
Incomplete factorization algorithms are built on top of a classical matrix factorization algorithm. Incomplete LU (ILU) 
for instance starts with a classical LU algorithm and ignores some of the fill-in based on thresholding and on an artificially
prescribed maximum number of non-zeros in every row \& column \cite{saad1994ilut}. 
Block versions \cite{saad1999bilum} are sometimes used because of better robustness (with possible pivoting)
and practical properties (cache-friendly algorithm, use of BLAS, etc.). Once an incomplete 
LU factorization has been computed, it can be used as a preconditioner for a CG of GMRES algorithm for instance.

Matrices arising from PDE discretization also typically have low-rank off-diagonal blocks \cite{bebendorf2003existence, 
bebendorf2005efficient, chandrasekaran2010numerical}. More precisely, the fill-in arising when factoring the matrix 
typically has small numerical rank, with weak dependence on $N$. 
This is closely related to the existence of a smooth Green's function for the underlying PDE and to the Fast Multipole 
Method \cite{barnes1986hierarchical, Greengard1987, fong2009black}.
Matrices built using this property are broadly called Hierarchical ($\H$) matrices \cite{hackbusch1999sparse}.
Many formats exist, depending on when and how off-diagonal blocks are compressed into low-rank format. The Hierarchical 
Off-Diagonal Low Rank (HODLR) \cite{ambikasaran2013fast} format compresses all off-diagonal blocks. If the off-diagonal 
are compressed using a nested basis, we obtain Hierarchically Semi-Separable (HSS) matrices 
\cite{chandrasekaran2004fast, chandrasekaran2005some, chandrasekaran2006fast, xia2010fast}. 
Finally, the broader category of $\H^2$ matrices also uses nested basis but only compresses well-separated (i.e., far-field) 
interactions (\cite{hackbusch2002data, hackbusch2002h2, yang2016sparse}, \cite{pouransari2017fast} with LoRaSp \rev{and \cite{sushnikova2018compress} with the ``Compress and Eliminate'' solver).}
All of those representations lead to a data-sparse representation of the matrix with tunable accuracy (by making the low-rank approximations more or less
accurate) and fast inverse computations.
This can then be used to construct preconditioners.
These constructions, while asymptotically efficient, sometimes have fairly large constants.

Attempts to improve the practical performance rely on exploiting sparsity as well as the low rank structure.
Most approaches up to date have focused on incorporating fast (i.e., $\H$-) algebra into the classical Nested Dissection
algorithm \cite{grasedyck2009domain} in order to decrease the cost of working with large fronts.
Other works have taken the similar approach of incorporating rank structured matrices into a multifrontal factorization 
in order to compress the large dense frontal matrices. 
\rev{HSS is often used to compress the large frontal matrices} \cite{xia2009superfast, schmitz2012fast, xia2013efficient, xia2013randomized, ghysels2016efficient}.
The last one was incorporated into the Strumpack package.
\cite{amestoy2015improving} uses Block Low-Rank approximation to compress the frontal matrices in the MUMPS solver 
\cite{amestoy2000mumps}. Finally, \cite{faverge:hal-01187882} studies the use of $\H$-matrices using HODLR in the
PaStiX solver \cite{henon2002pastix}.

The Hierarchical Interpolative Factorization (HIF) \cite{ho2016hierarchical} proposes a different approach. 
Instead of storing the full dense fronts in some low-rank format, it uses low-rank approximation to directly sparsify 
(i.e., eliminate part of) the Nested Dissection separators without introducing any fill-in. 
As a result, the algorithm never deals with large edges (in low-rank format or not) but instead constantly reduces
the size of all edges and separators.
This is the approach we take.

\revrev{We finally mention some recent work by J. Xia \& Z. Xin \cite{xia2017effective} and J. Feliu-Fab{\`a} et al. \cite{feliu2018recursively} where, in both cases, a scale-then-compress approach is taken.
Our algorithm shares similarities with those, as we also scale the matrix block using the Cholesky factorization of the
pivot}. As we will see, this significantly improves the preconditioner's accuracy.

\subsection{Contribution}

Our approach is based on the idea of HIF described in \cite{ho2016hierarchical}. 
However, there are several differences, improvements and novel capabilities:
\begin{itemize}
    \item Our algorithm is completely general and can be applied to any (SPD) matrix. 
        The only required input is the sparse matrix itself. 
        If geometry information is available, it can be used to improve the quality of the ordering and clustering.
    \item We incorporate an additional diagonal block scaling step in the algorithm, greatly improving the accuracy
of the preconditioner for only a small additional cost;
    \item We use an orthogonal (instead of interpolative) transformation, improving stability and guaranteeing that the
preconditioner stays SPD when A is SPD;
    \item We test the algorithm on more and larger test problems.
\end{itemize}

In a nutshell, our algorithm is based on a couple of key ideas. First, we start with a nested dissection (ND) ordering.
Then following the idea introduced in \cite{ho2016hierarchical}, after each elimination step, we sparsify the interfaces
between just-eliminated interiors, effectively reducing the size of \emph{all} ND separators. 
This is done using low-rank approximation, allowing to sparsify the separators without introducing any fill-in.
We then merge clusters and proceed to the next level.

A natural consequence of the above algorithm is that, if the compression fails to reduce the size of the separators, the
algorithm reverts to a (slower, but still relatively efficient) Nested Dissection algorithm.

\subsection{Contrast with fast algebra based algorithms}

We emphasize that the HIF approach \cite{ho2016hierarchical} \rev{and ours} are different from the classical way of accelerating sparse direct solvers.
Consider for instance the top separator of a Nested Dissection elimination.
At the end of the elimination, the corresponding (very large) block in the matrix is typically dense.
MUMPS \cite{amestoy2015improving} and PaStiX \cite{faverge:hal-01187882} for instance use fast $\H$-algebra to compress this block. 
This allows for fast factorization, inversion, etc.

As indicated above, we take a different approach. 
Instead of storing large blocks (corresponding to large separators) in low-rank format (typically using $\H$-matrices), we eliminate part 
of the separators \rev{\emph{right from the beginning}}, effectively reducing their size.
We do so without introducing any fill-in, but at the expense of an approximate factorization.
As a result, the top separator \rev{remains} dense but is much smaller than at the beginning.

\rev{Both approaches use some sort of hierarchical clustering of the unknowns. The difference lies in the order of operations.
In the first category (large blocks using fast $\H$-algebra) elimination is delayed until the end. The result are large and dense but hierarchically low-rank fronts.
In our approach (like in \cite{ho2016hierarchical}), fronts are kept small throughout the factorization by eliminating unknowns related to low-rank interactions as soon as possible.}

\subsection{Organization of the paper}

This paper is organized in three sections. First, \autoref{sec:spand} introduces and motivates the algorithm, starting at 
a high level and later introducing the details. Then, \autoref{sec:theoretical_results} proves the stability of the scheme, discusses
the choice of the low-rank approximation \rev{and provides a complexity analysis}.
Finally, \autoref{sec:numerical_experiments} shows numerical experiments on medium and large scale matrices.

\section{Sparsified Nested Dissection} \label{sec:spand}

This section describes the algorithm in detail.
We start by discussing Nested Dissection and some of its characteristics. Then, building upon it, we introduce our 
algorithm, and then detail all the various parts.

\subsection{Classical Nested Dissection Ordering} \label{sec:nested_dissection_ordering}

Nested Dissection (ND) \cite{george1973nested} is an ordering strategy to factor sparse matrices.
Consider a sparse symmetric matrix $A \in \mathbb{R}^{N\times N}$ and its graph $G_A = (V,E)$ defined as $V = \{1,\dots,N\}$ 
and $E = \{(i,j): A_{ij} \neq 0\}$. Notice that since $A$ is symmetric, the graph is undirected.
The basic building block of ND is the computation of \emph{vertex-separators}. 
Starting with the full graph, one finds a cluster of vertices, a vertex-separator, separating the graph into two disconnected 
clusters (a cluster is a subset of $V$). 

\autoref{sfig:classical_ND_ordering_b} gives an example of such a separator.
The idea is then applied recursively as indicated in \autoref{sfig:classical_ND_ordering_c} and \autoref{sfig:classical_ND_ordering_d}.
That is, disconnected clusters are further sub-divided using separators. 
This recursive process is repeated until cluster sizes are small enough to be factored using some direct dense method. 
A matrix factorization can begin by eliminating unknowns in all disconnected clusters defined by the last recursive level of 
the nested dissection process. 
Thus, the only non-eliminated unknowns correspond to degrees-of-freedom (dofs) associated with all of the separators. 
We then proceed to eliminate all dofs associated with the last set of separators (e.g., those defined in the $\ell = 3$ level of \autoref{fig:classical_nested_dissection}). 
Once these have been eliminated, we proceed by eliminating the second to last set of separator unknowns (e.g., those defined on $\ell =2$ in \autoref{fig:classical_nested_dissection}). 
The process continues eliminating unknowns associated with successively lower levels. 
This process can be viewed as an elimination tree, illustrated in \autoref{fig:classical_nested_dissection_etree}.

The elimination tree indicates dependencies between operations. Each node is a cluster in the graph of $A$ (separator or leaf-interior), 
and a cluster can only be eliminated once all its \rev{descendants} have been eliminated. 
The clusters are then eliminated from bottom to top.
This follows from the fact that eliminating a parent before a child would create edges between clusters 
previously separated, breaking the purpose of the ordering.
ND is an ordering that limits fill-in: by eliminating clusters from bottom to top, one never creates edges 
(i.e., fill-in) between clusters previously separated.

\begin{figure}
\centering
    %%%%
    \subfloat[\label{sfig:classical_ND_ordering_a}$\ell = 0$]{\begin{tikzpicture}
        \def \l{1.3};
        \draw[fill=white,draw=none] (-\l,\l) rectangle (\l,3*\l);
        \begin{scope}[every node/.style={draw,circle}]
            \node (a) at (0,2*\l) {};
        \end{scope}
        \draw [rounded corners=3pt] (-\l,-\l) rectangle (\l,\l);
    \end{tikzpicture}} \,
    %%%%
    \subfloat[\label{sfig:classical_ND_ordering_b}$\ell = 1$]{\begin{tikzpicture}
        \def \l{1.3};
        \draw[fill=white,draw=none] (-\l,\l) rectangle (\l,3*\l);
        \begin{scope}[every node/.style={draw,circle}]
            \node (a) at (-0.7*\l,2*\l) {};
            \node [fill=black!80] (b) at (0,2*\l) {};
            \node (c) at (0.7*\l,2*\l) {};
        \end{scope}
        \begin{scope}
            \draw[-] (a) edge (b);
            \draw[-] (b) edge (c);
        \end{scope}
        \draw [rounded corners=3pt] (-\l,-\l) rectangle (\l,\l);
        % 0
        \draw [rounded corners=3pt,fill=black!80] (-0.075*\l,-\l) rectangle (0.075*\l,\l);
    \end{tikzpicture}} \,
    %%%%
    \subfloat[\label{sfig:classical_ND_ordering_c}$\ell = 2$]{\begin{tikzpicture}
        \def \l{1.3};
        \draw[fill=white,draw=none] (-\l,\l) rectangle (\l,3*\l);
        \begin{scope}[every node/.style={draw,circle}]
            \node[fill=black!80] (a) at (0 ,2*\l) {};
            \node[fill=black!50] (b) at (-0.7*\l,2*\l) {};
            \node[fill=black!50] (c) at (0.7*\l,2*\l) {};
            \node (d) at (-0.7*\l,2.3*\l) {};
            \node (e) at (-0.7*\l,1.7*\l) {};
            \node (f) at (0.7*\l,2.3*\l) {};
            \node (g) at (0.7*\l,1.7*\l) {};
        \end{scope}
        \begin{scope}
            \draw [-] (a) edge (b);
            \draw [-] (a) edge (c);
            \draw [-] (a) edge (d);
            \draw [-] (a) edge (e);
            \draw [-] (a) edge (f);
            \draw [-] (a) edge (g);
            \draw [-] (b) edge (d);
            \draw [-] (b) edge (e);
            \draw [-] (c) edge (f);
            \draw [-] (c) edge (g);
        \end{scope}
        \draw [rounded corners=3pt] (-\l,-\l) rectangle (\l,\l);
        % 0
        \draw [rounded corners=3pt,fill=black!80] (-0.075*\l,-\l) rectangle (0.075*\l,\l);
        % -1
        \draw [rounded corners=3pt,fill=black!50] (-\l,0.1*\l) rectangle (-0.075*\l,0.25*\l);
        \draw [rounded corners=3pt,fill=black!50] (0.075*\l,-0.15*\l) rectangle (\l,-0.3*\l);
    \end{tikzpicture}} \,
    %%%%
    \subfloat[\label{sfig:classical_ND_ordering_d}$\ell = 3$]{\begin{tikzpicture}
        \def \l{1.3};
        \draw[fill=white,draw=none] (-\l,\l) rectangle (\l,3*\l);
        \begin{scope}[every node/.style={draw,circle}]
            % 0
            \node[fill=black!80] (a) at (0 ,2*\l) {};
            % 1
            \node[fill=black!50] (b) at (-0.7*\l,2*\l) {};
            \node[fill=black!50] (c) at (0.7*\l,2*\l) {};
            % 2
            \node[fill=black!25] (d) at (-0.7*\l,2.3*\l) {};
            \node[fill=black!25] (e) at (-0.7*\l,1.7*\l) {};
            \node[fill=black!25] (f) at (0.7*\l,2.3*\l) {};
            \node[fill=black!25] (g) at (0.7*\l,1.7*\l) {};
            % 3
            \node (h) at (-1.1*\l,2.3*\l) {};
            \node (i) at (-0.3*\l,2.3*\l) {};
            \node (j) at (-1.1*\l,1.7*\l) {};
            \node (k) at (-0.3*\l,1.7*\l) {};
            \node (l) at (1.1*\l,2.3*\l) {};
            \node (m) at (0.3*\l,2.3*\l) {};
            \node (n) at (0.3*\l,1.7*\l) {};
            \node (o) at (1.1*\l,1.7*\l) {};
        \end{scope}
        \begin{scope}
            \draw [-] (a) edge (b);
            \draw [-] (a) edge (c);
            \draw [-] (a) edge (d);
            \draw [-] (a) edge (e);
            \draw [-] (a) edge (f);
            \draw [-] (a) edge (g);
            \draw [-] (b) edge (d);
            \draw [-] (b) edge (e);
            \draw [-] (c) edge (f);
            \draw [-] (c) edge (g);
            \draw [-] (h) edge (a);
            \draw [-] (h) edge (b);
            \draw [-] (h) edge (d);
            \draw [-] (i) edge (a);
            \draw [-] (i) edge (b);
            \draw [-] (i) edge (d);
            \draw [-] (j) edge (a);
            \draw [-] (j) edge (b);
            \draw [-] (j) edge (e);
            \draw [-] (k) edge (a);
            \draw [-] (k) edge (b);
            \draw [-] (k) edge (a);
            \draw [-] (l) edge (a);
            \draw [-] (l) edge (c);
            \draw [-] (l) edge (f);
            \draw [-] (m) edge (a);
            \draw [-] (m) edge (c);
            \draw [-] (m) edge (f);
            \draw [-] (n) edge (a);
            \draw [-] (n) edge (c);
            \draw [-] (n) edge (g);
            \draw [-] (o) edge (a);
            \draw [-] (o) edge (c);
            \draw [-] (o) edge (g);
        \end{scope}
        \draw [rounded corners=3pt] (-\l,-\l) rectangle (\l,\l);
        % 0
        \draw [rounded corners=3pt,fill=black!80] (-0.075*\l,-\l) rectangle (0.075*\l,\l);
        % -1
        \draw [rounded corners=3pt,fill=black!50] (-\l,0.1*\l) rectangle (-0.075*\l,0.25*\l);
        \draw [rounded corners=3pt,fill=black!50] (0.075*\l,-0.15*\l) rectangle (\l,-0.3*\l);
        % -2
        \draw [rounded corners=3pt,fill=black!25] (-\l,0.5*\l) rectangle (-0.075*\l,0.65*\l);
        \draw [rounded corners=3pt,fill=black!25] (0.5*\l,-0.15*\l) rectangle (0.65*\l,\l);
        \draw [rounded corners=3pt,fill=black!25] (-\l,-0.55*\l) rectangle (-0.075*\l,-0.7*\l);
        \draw [rounded corners=3pt,fill=black!25] (0.075*\l,-0.55*\l) rectangle (\l,-0.7*\l);
    \end{tikzpicture}} 
    \caption{Classical Nested Dissection ordering.}
    \label{fig:classical_nested_dissection}
\end{figure}
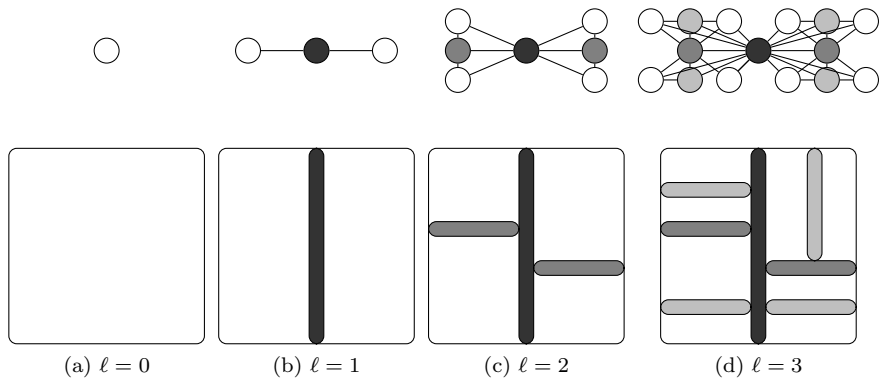
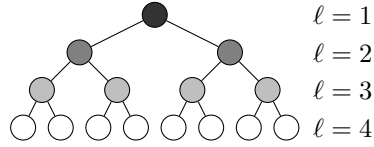
\begin{figure}
    \centering
    \begin{tikzpicture}[auto, node distance=3cm]
        \begin{scope}[every node/.style={draw,circle}]
            \node[fill=black!80] (a) at (0 ,0) {};
            \node[fill=black!50] (b) at (-1,-0.5) {};
            \node[fill=black!50] (c) at (1,-0.5) {};
            \node[fill=black!25] (d) at (-1.5,-1) {};
            \node[fill=black!25] (e) at (-0.5,-1) {};
            \node[fill=black!25] (f) at (0.5,-1) {};
            \node[fill=black!25] (g) at (1.5,-1) {};
            \node (h) at (-1.75,-1.5) {};
            \node (i) at (-1.25,-1.5) {};
            \node (j) at (-0.75,-1.5) {};
            \node (k) at (-0.25,-1.5) {};
            \node (l) at (0.25,-1.5) {};
            \node (m) at (0.75,-1.5) {};
            \node (n) at (1.25,-1.5) {};
            \node (o) at (1.75,-1.5) {};
            \node[draw=none] at (2.5,0) {$\ell=1$};
            \node[draw=none] at (2.5,-0.5) {$\ell=2$};
            \node[draw=none] at (2.5,-1) {$\ell=3$};
            \node[draw=none] at (2.5,-1.5) {$\ell=4$};
        \end{scope}
        \begin{scope}
            \draw [-] (a) edge (b);
            \draw [-] (a) edge (c);
            \draw [-] (b) edge (d);
            \draw [-] (b) edge (e);
            \draw [-] (c) edge (f);
            \draw [-] (c) edge (g);
            \draw [-] (d) edge (h);
            \draw [-] (d) edge (i);
            \draw [-] (e) edge (j);
            \draw [-] (e) edge (k);
            \draw [-] (f) edge (l);
            \draw [-] (f) edge (m);
            \draw [-] (g) edge (n);
            \draw [-] (g) edge (o);
        \end{scope}
    \end{tikzpicture}
    \caption{The elimination tree associated to the ND ordering on \autoref{fig:classical_nested_dissection}}
    \label{fig:classical_nested_dissection_etree}
\end{figure}

The elimination procedure can also be represented in matrix form.
Denote the total number of levels by $L$ (where the leaves correspond to $\ell = L$ and the root to $\ell = 1$).
Define $A^{(L)}$ as the entire matrix and let \rev{$A^{(\ell)}$}  (for $\ell < L$) be the Schur complement operator (trailing matrix) obtained by eliminating levels $\ell+1,\dots,L$. 
The matrix obtained after eliminating the $\ell+1$ level can be written in a block-arrowhead form
\[ A^{(\ell)} =    \begin{bmatrix}      A^{(\ell)}_{11} &                &                   & A^{(\ell)}_{1q} \\
                                                        &  \ddots        &                   & \vdots \\
                                                        &                & A^{(\ell)}_{mm}   & A^{(\ell)}_{mq} \\
                                        A^{(\ell)}_{q1} &  \dots         & A^{(\ell)}_{qm}   & A^{(\ell)}_{qq} \end{bmatrix} \]
where $m = 2^{\ell-1}$ and $q = m+1$. Here,  $A^{(\ell)}_{qq}$ refers to the matrix associated with all separators at levels $1,\dots,\ell-1$.
The $\rev{A^{(\ell)}_{ii}}$ (for $i \le m$) are the matrices associated with non-eliminated unknowns within the $i^{th}$ disconnected separators on the $\ell^{th}$ level. 
The Schur complement can now be written as
\[ A^{(\ell-1)} = A^{(\ell)}_{qq} - \sum_{i=1}^m \rev{A^{(\ell)}_{qi}}   \left(\rev{A^{(\ell)}_{ii}}\right)^{-1}  \rev{A^{(\ell)}_{iq}}. \]
This new matrix can then be interpreted as another block-arrowhead matrix associated with level $\ell-1$ and so the elimination procedure can be repeated. 

While limited, the fill-in is still significant. For instance, once all descendants of the top separators have 
been eliminated, the top separator is typically completely filled (dense).
For problems arising from the discretization of PDE's in 3D with $\OO{N} = \OO{n^3}$ degrees of freedom (dofs), 
the top separator typically has size $\OO{N^{2/3}} = \OO{n^2}$. For instance in a regular $n \times n \times n$ cube with 
$N = n^3$ dofs and a 7-points stencil (or any other stencil with only ``local'' connections), the top separator is a 
plane (see \autoref{fig:nested_dissection_3D}) of size $n \times n = n^2 = N^{2/3}$. Hence, its factorization will cost 
$\OO{N^{2/3 \times 3}} = \OO{N^2}$, leading to quadratic or near-quadratic algorithms. While this is only formally valid 
on regular cubic-shaped graphs, the issue extends beyond those problems \cite{lipton1979generalized}: the separators in 
3D graphs are typically very large, leading to large Schur complements and an expensive factorization, with complexity well above $\OO{N}$.
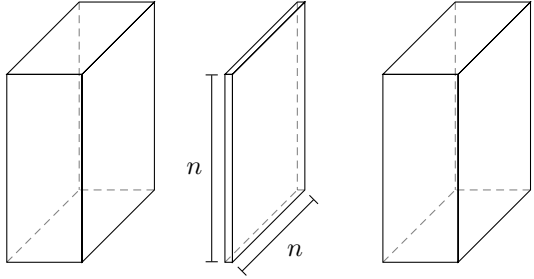
\begin{figure}
\centering
\begin{tikzpicture}[every edge quotes/.append style={auto}]
    \pgfmathsetmacro{\cubex}{1}
    \pgfmathsetmacro{\cubey}{2.5}
    \pgfmathsetmacro{\cubez}{2.5}
    \draw [draw=black, every edge/.append style={draw=black, densely dashed, opacity=.5}]
      (0,0,0) coordinate (o) -- ++(-\cubex,0,0) coordinate (a) -- ++(0,-\cubey,0) coordinate (b) edge coordinate [pos=1] (g) ++(0,0,-\cubez)  -- ++(\cubex,0,0) coordinate (c) -- cycle
      (o) -- ++(0,0,-\cubez) coordinate (d) -- ++(0,-\cubey,0) coordinate (e) edge (g) -- (c) -- cycle
      (o) -- (a) -- ++(0,0,-\cubez) coordinate (f) edge (g) -- (d) -- cycle;
    \draw [draw=black, every edge/.append style={draw=black, densely dashed, opacity=.5}]
      (5,0,0) coordinate (o) -- ++(-\cubex,0,0) coordinate (a) -- ++(0,-\cubey,0) coordinate (b) edge coordinate [pos=1] (g) ++(0,0,-\cubez)  -- ++(\cubex,0,0) coordinate (c) -- cycle
      (o) -- ++(0,0,-\cubez) coordinate (d) -- ++(0,-\cubey,0) coordinate (e) edge (g) -- (c) -- cycle
      (o) -- (a) -- ++(0,0,-\cubez) coordinate (f) edge (g) -- (d) -- cycle;
    \pgfmathsetmacro{\cubex}{0.1}
    \pgfmathsetmacro{\cubey}{2.5}
    \pgfmathsetmacro{\cubez}{2.5}
    \draw [draw=black, every edge/.append style={draw=black, densely dashed, opacity=.5}]
      (2,0,0) coordinate (o) -- ++(-\cubex,0,0) coordinate (a) -- ++(0,-\cubey,0) coordinate (b) edge coordinate [pos=1] (g) ++(0,0,-\cubez)  -- ++(\cubex,0,0) coordinate (c) -- cycle
      (o) -- ++(0,0,-\cubez) coordinate (d) -- ++(0,-\cubey,0) coordinate (e) edge (g) -- (c) -- cycle
      (o) -- (a) -- ++(0,0,-\cubez) coordinate (f) edge (g) -- (d) -- cycle;
    \path [every edge/.append style={draw=black, |-|}]
      (b) +(-5pt,0) coordinate (b2) edge ["$n$"] (b2 |- a)
      (c) +(3.5pt,-3.5pt) coordinate (c2) edge ["$n$"'] ([xshift=3.5pt,yshift=-3.5pt]e);
\end{tikzpicture}
\caption{Classical ND in $3D$ with $N = n^3$ nodes: the top separator is of size $\OO{n^2}$. 
Once the left and right clusters have been eliminated the top separator becomes completely dense, 
making its elimination alone cost $\OO{(n^2)^3} = \OO{N^2}$.}
\label{fig:nested_dissection_3D}
\end{figure}

Our algorithm addresses this specific concern by continually decreasing the size of all separators to 
keep fill-in to a minimum. It does so using low-rank approximations, and the factorization is then only approximate.
In most cases under consideration, the separator size is typically decreased to $\OO{n} = \OO{N^{1/3}}$ so that 
its factorization costs $\OO{N}$.

\subsection{Sparsified Nested Dissection}
\label{sec:sparsification_high_level}

As noted, the sub-blocks created by the repeated Schur complement process become denser.  
To further limit fill-in, we introduce a sparsification algorithm that is invoked after all eliminations associated with a particular level have been performed. 
To motivate the sparsification, let us first consider a very simple case. Suppose \rev{we} are about to eliminate level $\ell$ and 
that there exist a subset of the $j$\textsuperscript{th} separator such that the corresponding rows in $A_{jq}^{(\ell)}$ are relatively small.
An inexact or incomplete factorization can be defined by simply ignoring these small rows. 
This effectively decouples those unknowns from the rest of the system so that they can be eliminated right away, 
without causing an increase in the number of non-zeros in the next recursive Schur complement. 
We can think of this as decreasing the size of the $j$\textsuperscript{th} separator.  
While $A^{(\ell)}_{jq}$  will not generally have small rows, we instead seek a transformation that produces the desired small 
rows without altering the nonzero structure of the matrix.
Further, this transformation will not only be applied to the off-diagonal blocks in the level $\ell$ arrowhead matrix, 
but to all the non-eliminated degrees of freedom (dofs) (including those in $A^{(\ell)}_{qq})$. 
This \rev{means} we will effectively decrease the size of all remaining separators at levels $1,\dots,\ell$.

\begin{figure}
    \subfloat[The separators\label{sfig:separator_sparsification_process_a}]{
        \begin{tikzpicture}
            \def \l{1.3};
            \draw [rounded corners=3pt,fill=black!40] (-\l,-\l) rectangle (-0.075*\l,0.15*\l);
            \draw [rounded corners=3pt,fill=black!40] (-\l,0.30*\l) rectangle (-0.075*\l,\l);
            \draw [rounded corners=3pt,fill=black!40] (0.075*\l,-\l) rectangle (\l,-0.5*\l);
            \draw [rounded corners=3pt,fill=black!40] (0.075*\l,-0.35*\l) rectangle (\l,\l);
            % 0
            \draw [rounded corners=3pt] (-0.075*\l,-\l) rectangle (0.075*\l,\l);
            % -1
            \draw [rounded corners=3pt] (-\l,0.15*\l) rectangle (-0.075*\l,0.3*\l);
            \draw [rounded corners=3pt] (0.075*\l,-0.5*\l) rectangle (\l,-0.35*\l);
        \end{tikzpicture}
    }\;
    \subfloat[The interfaces\label{sfig:separator_sparsification_process_b}]{
        \begin{tikzpicture}
            \def \l{1.3};
            \draw [rounded corners=3pt,fill=black!40] (-\l,-\l) rectangle (-0.075*\l,0.15*\l);
            \draw [rounded corners=3pt,fill=black!40] (-\l,0.30*\l) rectangle (-0.075*\l,\l);
            \draw [rounded corners=3pt,fill=black!40] (0.075*\l,-\l) rectangle (\l,-0.5*\l);
            \draw [rounded corners=3pt,fill=black!40] (0.075*\l,-0.35*\l) rectangle (\l,\l);
            % 0
            \draw [rounded corners=3pt] (-0.075*\l,-\l) rectangle (0.075*\l,-0.5*\l);
            \draw [rounded corners=3pt] (-0.075*\l,-0.5*\l) rectangle (0.075*\l,-0.35*\l);
            \draw [rounded corners=3pt] (-0.075*\l,-0.35*\l) rectangle (0.075*\l,0.15*\l);
            \draw [rounded corners=3pt] (-0.075*\l,0.15*\l) rectangle (0.075*\l,0.3*\l);
            \draw [rounded corners=3pt] (-0.075*\l,0.3*\l) rectangle (0.075*\l,\l);
            % -1
            \draw [rounded corners=3pt] (-\l,0.15*\l) rectangle (-0.075*\l,0.3*\l);
            \draw [rounded corners=3pt] (0.075*\l,-0.5*\l) rectangle (\l,-0.35*\l);
        \end{tikzpicture}
    }\;
    \subfloat[][Before sparsification\label{sfig:separator_sparsification_process_c}]{\begin{tikzpicture}[every node/.style={circle,draw,inner sep=0pt,text width=5pt}]
        \def \l{1.3};
        % Create nodes
        \foreach \y in {-4,...,4}
           \node (0\y) at (0,0.25*\y*\l) {};
        \foreach \x in {-4,...,-1}
           \node (\x0) at (0.25*\x*\l,0) {};
        \foreach \x in {1,...,4}
           \node (\x0) at (0.25*\x*\l,-0.5*\l) {};
        % Draw edges
        \foreach \x in {-4,...,4}
            \foreach \y in {-4,...,4}
                \draw [-] (\x0) edge [black!50] (0\y);
        % Draw rectangle
        \draw [rounded corners=2pt,fill=white] (-0.125*\l,-1.125*\l) rectangle (0.125*\l,1.125*\l);
        \draw [rounded corners=2pt,fill=white] (-1.125*\l,-0.125*\l) rectangle (-0.125*\l,0.125*\l);
        \draw [rounded corners=2pt,fill=white] (0.125*\l,-0.625*\l) rectangle (1.125*\l,-0.375*\l);
        % Add labels
        \begin{scope}[every node/.style={}]
            \node at (0.25*\l,\l) {$p$};
            \node at (-\l,0.25*\l)  {$n$};
        \end{scope}
        % Fill nodes
        \foreach \y in {1,...,4}
            \node[fill=black!75] at (0,0.25*\y*\l) {};
        \foreach \y in {-2,...,0}
            \node[fill=black!25] at (0,0.25*\y*\l) {};
        \foreach \x in {-4,...,-1}
            \node[fill=black!25] at (0.25*\x*\l,0) {};
        \foreach \x in {1,...,4}
            \node[fill=black!25] at (0.25*\x*\l,-0.5*\l) {};
        \foreach \y in {-4,...,-1}
            \node at (0,0.25*\y*\l) {};
        \end{tikzpicture}
    }\;
    \subfloat[After sparsification\label{sfig:separator_sparsification_process_d}]{\begin{tikzpicture}[every node/.style={circle,draw,inner sep=0pt,text width=5pt}]
        \def \l{1.3};
        % Create nodes
        \foreach \y in {-4,...,4}
            \node (0\y) at (0,0.25*\y*\l) {};
        \foreach \x in {-4,...,-1}
            \node (\x0) at (0.25*\x*\l,0) {};
        \foreach \x in {1,...,4}
            \node (\x0) at (0.25*\x*\l,-0.5*\l) {};
        % Draw edges
        \foreach \x in {-4,...,4}
            \foreach \y in {-4,...,0}
                \draw [-] (\x0) edge [black!50] (0\y);
        \foreach \x in {-4,...,4}
            \foreach \y in {1,...,2}
                \draw [-] (\x0) edge [thick] (0\y);
        % Draw rectangle
        \draw [rounded corners=2pt,fill=white] (-0.125*\l,-1.125*\l) rectangle (0.125*\l,1.125*\l);
        \draw [rounded corners=2pt,fill=white] (-1.125*\l,-0.125*\l) rectangle (-0.125*\l,0.125*\l);
        \draw [rounded corners=2pt,fill=white] (0.125*\l,-0.625*\l) rectangle (1.125*\l,-0.375*\l);
        % Add labels
        \begin{scope}[every node/.style={}]
            \node at (0.25*\l,\l) { $f$ };
            \node at (0.25*\l,0.5*\l) { $c$ };
            \node at (-\l,0.25*\l) { $n$ };
        \end{scope}
        % Fill nodes
        \foreach \y in {3,...,4}
            \node[pattern=north west lines] at (0,0.25*\y*\l) {};
        \foreach \y in {1,...,2}
            \node[pattern=north east lines] at (0,0.25*\y*\l) {};
        \foreach \y in {-2,...,0}
            \node[fill=black!25] at (0,0.25*\y*\l) {};
        \foreach \x in {-4,...,-1}
            \node[fill=black!25] at (0.25*\x*\l,0) {};
        \foreach \x in {1,...,4}
            \node[fill=black!25] at (0.25*\x*\l,-0.5*\l) {};
        \foreach \y in {-4,...,-1}
            \node[draw,circle] at (0,0.25*\y*\l) {};
        \end{tikzpicture}
    }     
    \caption{Separator sparsification process. The first picture depicts the usual ND separator.
    The grey boxes have been eliminated. The second picture shows the different interfaces to be sparsified.
    The third and fourth picture focus on a given interface defined by $p$ and connected to $n$.
    On the fourth picture, we have transformed $p$ into $f$ and $c$ through a change of basis (not shown here) and up to a small $\OO{\varepsilon}$ error.
    $f$ is now disconnected from $n$ and can be eliminated, without introducing any fill-in. Dark edges are edges updated by the sparsification.}
    \label{fig:separator_sparsification_process}
\end{figure}
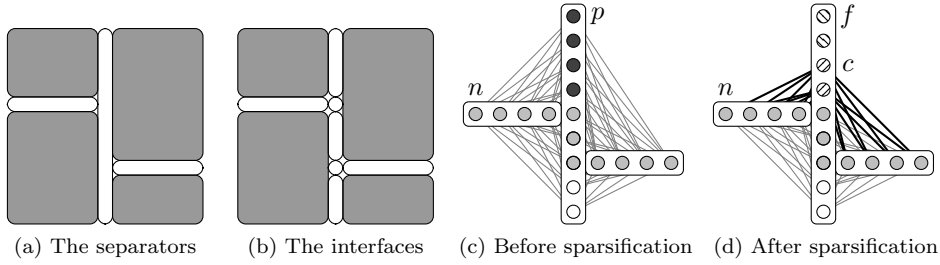

To understand the transformation, it is best to switch to a different block structure for the matrix. 
Specifically, consider the matrix
\[ A^{(\ell)} = \begin{bmatrix}
        \hat{A}^{(\ell)}_{11} & \dots     & \hat{A}^{(\ell)}_{1M}   \\
        \vdots                & \ddots    & \vdots                \\
        \hat{A}^{(\ell)}_{M1} & \dots     & \hat{A}^{(\ell)}_{MM}
    \end{bmatrix}
\]
This matrix is equivalent to $A^{(\ell)}$, the use of the {\it hat} accent symbol is only to emphasize the different block structure.  
Each $\hat{A}^{(\ell)}_{ii}$ corresponds to the sub-matrix associated with an \emph{interface} (instead of a separator).
Specifically, an interface is defined as a \emph{subset} of a separator for which the left and right neighbors correspond to a given pair of separators at level $\ell$.
These $M$ interfaces are subsets of all the separators associated with non-eliminated levels (i.e., levels $k$ with $k \le \ell$).  
Notice that many of the $\hat{A}^{(\ell)}_{ij}$ are zero as only neighboring interfaces are coupled.  
\autoref{fig:separator_sparsification_process} shows the distinction between ND separators (\autoref{sfig:separator_sparsification_process_a}) and interfaces (\autoref{sfig:separator_sparsification_process_b}).
The top-level (root) separator has been cut into 5 pieces, each associated to a pair of left and right neighbors.

\autoref{sfig:separator_sparsification_process_c} and \autoref{sfig:separator_sparsification_process_d} illustrates a typical situation 
and the effect of the sparsification.
Let $p$ be a subset of a ND separator (in dark grey) at the interface between two interiors and $n$ be all its neighbors (in light grey). 
The remaining nodes are disconnected from $p$ and can be ignored for the purpose of this discussion.
In this situation, the leaf-level interiors (dark grey clusters on \autoref{sfig:separator_sparsification_process_a} 
and \autoref{sfig:separator_sparsification_process_b}) have been eliminated
and only higher-levels separators are left.
The greyed edges represent connections between degrees of freedom. Notice that edges never cross separators.
 
Let $\hat A_{pn}$ \rev{denote} all the edges from $p$ to $n$. Assuming $\hat A_{pp} = I$ (this is not a restriction, see \autoref{sec:scaling}), 
we can then consider this sub-matrix of $A^{(\ell)}$
\[ \begin{bmatrix}
    I           & \hat A_{pn} \\
    \hat A_{np} & \hat A_{nn}    
\end{bmatrix} \]
Then, compute a low-rank approximation
\[ \hat A_{pn} = \underbrace{\begin{bmatrix} Q_{pf} & Q_{pc} \end{bmatrix}}_{Q} \begin{bmatrix} W_{fn} \\ W_{cn} \end{bmatrix}
\text{ with } \|W_{fn}\| = \OO{\varepsilon} \]
and verify that
\[ \begin{bmatrix} 
        Q^\top  &       \\ 
                & I 
    \end{bmatrix} 
    \begin{bmatrix} 
        I_{pp}      & \hat A_{pn} \\ 
        \hat A_{np} & \hat A_{nn} 
    \end{bmatrix}
    \begin{bmatrix} 
        Q   &   \\ 
            & I 
    \end{bmatrix} 
    = 
    \begin{bmatrix} 
        I_{ff}              &                       & \OO{\varepsilon} \\ 
                            & I_{cc}                & \rev{W_{cn}} \\ 
        \OO{\varepsilon}    & \rev{W_{cn}^\top}     & \hat A_{nn} 
    \end{bmatrix} 
\]
The matrix $Q$ is a change of variables, transforming $p$ into $f$ (``fine'') and $c$ (``coarse'' --- following AMG's 
terminology, \cite{stuben2001review}).
If we then ignore the $\OO{\varepsilon}$ edges, we have effectively decoupled the $f$ variables from the rest, i.e., 
we have eliminated $f$.
Notice that this didn't create any fill-in: $\hat A_{nn}$ is unchanged.
As a result, we effectively \emph{decreased the size of the separator}, without altering the nested dissection ordering.
In the following, we will drop the hat notation, as it should be clear from the context whether we are referring to separators or interfaces.

The algorithm then \rev{alternates} between classical ``interiors'' elimination (using a block Cholesky factorization) and ``interfaces'' 
sparsification as explained above. 
\autoref{algo:high_level_algo} presents a high-level version of the algorithm. We name the algorithm \algo{}, referring 
to ``sparsified Nested Dissection''.

\begin{algorithm}
    \begin{algorithmic}
        \REQUIRE{Sparse matrix $A$, SPD, Maximum level $L$}
        \STATE Compute a ND ordering for $A$, infer interiors, separators and interfaces (see \autoref{sec:ordering_and_clustering})
        \FORALL{$\ell=L,\dots,1$}
            \FORALL{$\mathcal{I}$ interior}
                \STATE Eliminate $\mathcal{I}$ (see \autoref{sec:elimination})
            \ENDFOR
            \FORALL{$\mathcal{B}$ interface between \revrev{interiors}}
                \STATE Sparsify $\mathcal{B}$ (see \autoref{sec:scaling} and \autoref{sec:sparsification})
            \ENDFOR
        \ENDFOR
    \end{algorithmic}
    \caption{High-level description of the \algo{} algorithm}
    \label{algo:high_level_algo}
\end{algorithm}

The subsequent sections explain in detail the ordering \& clustering (i.e., how we define the ``interfaces''), 
the elimination and sparsification.

\subsection{Ordering and Clustering}
\label{sec:ordering_and_clustering}

In addition to ordering, an appropriate clustering of the dofs has to be performed to define the various interfaces between interiors.
That is, a simple ND ordering, by itself, does not give any indication about what the interfaces between different interiors are.
To see this, consider \autoref{fig:classical_nested_dissection} (bottom row). This figure illustrates a classical ND 
ordering process. At every step, interiors are further separated by computing vertex separators. 
However, there is no clear way to define interfaces between interiors as shown on \autoref{fig:separator_sparsification_process} for instance.
This cannot readily be calculated or even properly defined with a ``usual'' ND ordering. 

To solve this issue, we have to keep track of the boundary of each interior during the ordering process.
\rev{We do so by modifying the usual ND algorithm. 
In the classical algorithm, a set of vertices is separated by a vertex-separator, and the algorithm then recurses on the ``left'' and ``right'' clusters (interiors).
We modify this by separating an interior \emph{and} its boundary using vertex separators. This \revrev{lets} keep track of the interfaces.
\autoref{fig:clustering_ordering} shows the high-level idea. 
For every interior $\mathcal{I}$ we keep track \revrev{of} its boundary $\mathcal{B}$ and we then separate their union $\mathcal{I} \cup \mathcal{B}$. 
}

\begin{figure}
    \centering
    \subfloat[][An interior and its boundary before the modified ND step.]{\begin{tikzpicture}
        \def \l{0.8};
        \draw (-2*\l,0) [white] rectangle (5.5*\l,3.5*\l);
        \draw (0,0) rectangle (3*\l,3*\l);
        \draw (0,0) rectangle (3.5*\l,3.5*\l);                
        \node at (1.5*\l,1.5*\l){\small $\mathcal{I}$};
        \node at (1.7*\l,3.25*\l){\small $\mathcal{B}$};
    \end{tikzpicture}} \;\;
    \subfloat[][The resulting left and right interiors and their boundaries.]{\begin{tikzpicture}
        \def \l{0.8};
        \draw (-2*\l,0) [white] rectangle (5.5*\l,3.5*\l);
        \draw [pattern=north east lines] (0,0) rectangle (3.5*\l,3.5*\l);
        \draw [fill=white] (0,0) rectangle (3*\l,3*\l);
        \draw [fill=white] (3*\l,1*\l) rectangle (3.5*\l,1.5*\l);
        \draw [fill=white] (3*\l,0) rectangle (3.5*\l,1*\l);
        \draw (0,1*\l) rectangle (3.5*\l,1.5*\l);
        \draw [pattern=north west lines] (0,1*\l) rectangle (3*\l,1.5*\l);
        \draw [pattern=north east lines] (0,1*\l) rectangle (3*\l,1.5*\l);
        \draw [pattern=north west lines] (3*\l,0) rectangle (3.5*\l,1*\l);
        \node at (1.5*\l,0.5*\l){\small $\mathcal{L} \cap \mathcal{I}$};
        \node at (1.5*\l,2.2*\l){\small $\mathcal{R} \cap \mathcal{I}$};
        \node at (-0.9*\l,1.3*\l){\small $\mathcal{M} \cap \mathcal{I}$};
        \node at (4.2*\l,2.5*\l){\small $\mathcal{R} \cap \mathcal{B}$};
        \node at (4.2*\l,0.5*\l){\small $\mathcal{L} \cap \mathcal{B}$};
    \end{tikzpicture}}
    \caption{The clustering \& ordering building block. On the left, an initial interior $\mathcal{I}$ and its boundary $\mathcal{B}$. 
    \rev{We then compute a vertex separator separating $\mathcal{I} \cup \mathcal{B}$ into left $\mathcal{L}$, right $\mathcal{R}$ and separator $\mathcal{M}$.}
    On the right, the resulting separated interiors and their boundaries, as well as the actual ND separator.}
    \label{fig:clustering_ordering}
\end{figure}
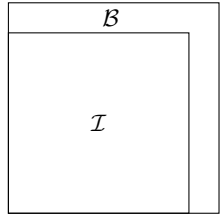
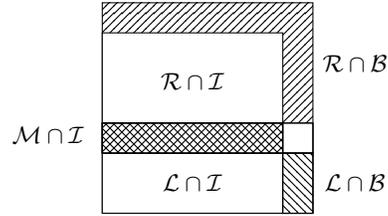

\rev{In practice, each node in the graph keeps track of its ``left'' and ``right'' neighboring separators, in addition to keeping track of the separator it belongs to.
We encode this by associating to each vertex $v$ a 3-tuple $(S, L, R)$.
$S$ is the usual ND separator $(\ell,k)$ where $\ell$ is its level and $1 \leq k \leq 2^{\ell-1}$.
$L$ and $R$ are the ND separators of $v$'s left and right neighbors, respectively.
\autoref{algo:ordering_clustering} formalizes this idea. Notice how the only building block is a vertex-separator routine, as available in Metis \cite{karypis1998fast}.
}
\begin{algorithm}
    \begin{algorithmic}
        \REQUIRE{$V$, vertices, $E$, edges, $L$ levels}
        \STATE \emph{\% Initialize the top separator (everyone), left and right neighbors (undefined)}
        \STATE $C[v] = (S:(1,1),L:\texttt{none},R:\texttt{none})$ for all $v \in V$
        \FORALL{$\ell = 1,\dots,L-1$}
            \FORALL{$k = 1,\dots,2^{\ell-1}$}
                \STATE \% \emph{Find interior to separate $\mathcal{I}$ and its boundary $\mathcal{B}$}
                \STATE $\mathcal{I} = \{v \in V: C[v]_S = (\ell,k)\}$
                \STATE $\mathcal{B} = \{v \in V: C[v]_L = (\ell,k) \text{ or } C[v]_R = (\ell,k)\}$
                \STATE \% \emph{Find vertex separator $\mathcal{M}$, left and right interiors $\mathcal{L}$ and $\mathcal{R}$}
                \STATE $(\mathcal{L},\mathcal{M},\mathcal{R}) = \texttt{vertex-separator}(\mathcal{I} \cup \mathcal{B})$
                \STATE \% \emph{Update separator, left and right interiors}
                \STATE $C[v]_S = (\ell,k)$ for all $v \in \mathcal{M} \backslash \mathcal{B}$
                \STATE $C[v]_S = (\ell+1,2k-1)$ for all $v \in \mathcal{L} \backslash \mathcal{B}$
                \STATE $C[v]_S = (\ell+1,2k)$ for all $v \in \mathcal{R} \backslash \mathcal{B}$
                \STATE \% \emph{Update neighbors of separator}
                \FORALL{$v \in \mathcal{M} \cap \mathcal{I}$}                    
                    \STATE $C[v]_L = (\ell+1,2k-1)$
                    \STATE $C[v]_R = (\ell+1,2k)$
                \ENDFOR 
                \STATE \% \emph{Update neighbors of left and right boundaries}
                \FORALL{$v \in \mathcal{L} \cap \mathcal{B}$}
                    \IF{$C[v]_L = (\ell,k)$}
                        \STATE $C[v]_L = (\ell+1,2k-1)$
                    \ELSE
                        \STATE $C[v]_R = (\ell+1,2k-1)$
                    \ENDIF
                \ENDFOR
                \FORALL{$v \in \mathcal{R} \cap \mathcal{B}$}
                    \IF{$C[v]_L = (\ell,k)$}
                        \STATE $C[v]_L = (\ell+1,2k)$
                    \ELSE
                        \STATE $C[v]_R = (\ell+1,2k)$
                    \ENDIF
                \ENDFOR                             
            \ENDFOR
        \ENDFOR
        \RETURN $C$
    \end{algorithmic}
    \caption{
        \rev{Ordering and clustering algorithm. 
        The algorithm is similar to a classical ND algorithm, except that it keeps track of the interfaces between interiors/separators, and recursively dissects interiors and their interfaces.
        ND separators are encoded as $(\ell,k)$ where $\ell$ is the level and $1 \leq k \leq 2^{\ell-1}$.}
        }
    \label{algo:ordering_clustering}
\end{algorithm}

\rev{\autoref{algo:ordering_clustering} returns $C$ that gives for each vertex $v$ in the graph its ND separator, $C[v]_S$, as well as a
tuple $(C[v]_L,C[v]_R)$ indicating its left and right neighboring interiors.
We then cluster together vertices $v$ with the same $C[v]$.
This algorithm is naturally recursive and defines, for each separator, a tree of clusters.}

\autoref{fig:ordering_clustering_process} illustrates the effect of \autoref{algo:ordering_clustering}
On the top row, we illustrate the separators at every step ($\ell$) of the algorithm.
The important distinctions with \autoref{fig:classical_nested_dissection} is that the computed vertex-separators overlap 
with the boundaries to keep track of interfaces, and each separator is further divided into clusters.
On the middle row, we illustrate the actual clusters at each level and how the ND separators are broken into pieces.
Separators at each level are depicted in a different color. 
Each separator is associated a hierarchy of clusters. 
The bottom row shows such a hierarchy within each separator and how those have to be merged when going from a lower to higher level.

\begin{figure}
    \centering
    % FIRST ROW
    \subfloat[][$\ell = 0$]{\begin{tikzpicture}
        \def \l{1.5};
        \draw [rounded corners=3pt] (-\l,-\l) rectangle (\l,\l);
    \end{tikzpicture}} \,
    \subfloat[][$\ell = 1$]{\begin{tikzpicture}
        \def \l{1.5};
        \draw [rounded corners=3pt] (-\l,-\l) rectangle (\l,\l);
        % 0
        \draw [rounded corners=3pt] (-0.075*\l,-\l) rectangle (0.075*\l,\l);
    \end{tikzpicture}} \,
    \subfloat[][$\ell = 2$]{\begin{tikzpicture}
        \def \l{1.5};
        \draw [rounded corners=3pt] (-\l,-\l) rectangle (\l,\l);
        % 0
        \draw [rounded corners=3pt] (-0.075*\l,-\l) rectangle (0.075*\l,\l);
        % -1
        \draw [rounded corners=3pt] (-\l,0.1*\l) rectangle (0.075*\l,0.25*\l);
        \draw [rounded corners=3pt] (-0.075*\l,-0.15*\l) rectangle (\l,-0.3*\l);
    \end{tikzpicture}} \,
    \subfloat[][$\ell = 3$]{\begin{tikzpicture}
        \def \l{1.5};
        \draw [rounded corners=3pt] (-\l,-\l) rectangle (\l,\l);
        % 0
        \draw [rounded corners=3pt] (-0.075*\l,-\l) rectangle (0.075*\l,\l);
        % -1
        \draw [rounded corners=3pt] (-\l,0.1*\l) rectangle (0.075*\l,0.25*\l);
        \draw [rounded corners=3pt] (-0.075*\l,-0.15*\l) rectangle (\l,-0.3*\l);
        % -2
        \draw [rounded corners=3pt] (-\l,0.5*\l) rectangle (0.075*\l,0.65*\l);
        \draw [rounded corners=3pt] (0.5*\l,-0.3*\l) rectangle (0.65*\l,\l);
        \draw [rounded corners=3pt] (-\l,-0.55*\l) rectangle (0.075*\l,-0.7*\l);
        \draw [rounded corners=3pt] (-0.075*\l,-0.55*\l) rectangle (\l,-0.7*\l);
    \end{tikzpicture}} \\
    % SECOND ROW
    \subfloat[][$\ell = 0$]{\begin{tikzpicture}
        \def \l{1.5};
        \draw [rounded corners=3pt] (-\l,-\l) rectangle (\l,\l);
    \end{tikzpicture}} \,
    \subfloat[][$\ell = 1$]{\begin{tikzpicture}
        \def \l{1.5};
        \draw [rounded corners=3pt] (-\l,-\l) rectangle (\l,\l);
        % 0
        \draw [rounded corners=3pt,fill=black!75] (-0.075*\l,-\l) rectangle (0.075*\l,\l);
    \end{tikzpicture}} \,
    \subfloat[][$\ell = 2$]{\begin{tikzpicture}
        \def \l{1.5};
        \draw [rounded corners=3pt] (-\l,-\l) rectangle (\l,\l);
        % 0
        \draw [rounded corners=3pt,fill=black!50] (-0.075*\l,-\l) rectangle (0.075*\l,-0.3*\l);
        \draw [rounded corners=3pt,fill=black!50] (-0.075*\l,-0.3*\l) rectangle (0.075*\l,-0.15*\l);
        \draw [rounded corners=3pt,fill=black!50] (-0.075*\l,-0.15*\l) rectangle (0.075*\l,0.1*\l);
        \draw [rounded corners=3pt,fill=black!50] (-0.075*\l,0.1*\l) rectangle (0.075*\l,0.25*\l);
        \draw [rounded corners=3pt,fill=black!50] (-0.075*\l,0.25*\l) rectangle (0.075*\l,\l);
        % -1
        \draw [rounded corners=3pt,pattern=north west lines] (-\l,0.1*\l) rectangle (-0.075*\l,0.25*\l);
        \draw [rounded corners=3pt,pattern=north west lines] (0.075*\l,-0.15*\l) rectangle (\l,-0.3*\l);
    \end{tikzpicture}} \;
    \subfloat[][$\ell = 3$]{\begin{tikzpicture}
        \def \l{1.5};
        \draw [rounded corners=3pt] (-\l,-\l) rectangle (\l,\l);
        % 0 bottom-top
        \draw [rounded corners=3pt,fill=black!20] (-0.075*\l,-\l) rectangle (0.075*\l,-0.70*\l);
        \draw [rounded corners=3pt,fill=black!20] (-0.075*\l,-0.7*\l) rectangle (0.075*\l,-0.55*\l);
        \draw [rounded corners=3pt,fill=black!20] (-0.075*\l,-0.55*\l) rectangle (0.075*\l,-0.3*\l);
        \draw [rounded corners=3pt,fill=black!20] (-0.075*\l,-0.3*\l) rectangle (0.075*\l,-0.15*\l);
        \draw [rounded corners=3pt,fill=black!20] (-0.075*\l,-0.15*\l) rectangle (0.075*\l,0.1*\l);
        \draw [rounded corners=3pt,fill=black!20] (-0.075*\l,0.1*\l) rectangle (0.075*\l,0.25*\l);
        \draw [rounded corners=3pt,fill=black!20] (-0.075*\l,0.25*\l) rectangle (0.075*\l,0.5*\l);
        \draw [rounded corners=3pt,fill=black!20] (-0.075*\l,0.5*\l) rectangle (0.075*\l,0.65*\l);
        \draw [rounded corners=3pt,fill=black!20] (-0.075*\l,0.65*\l) rectangle (0.075*\l,\l);
        % -1
        \draw [rounded corners=3pt,pattern=grid] (-\l,0.1*\l) rectangle (-0.075*\l,0.25*\l);
        \draw [rounded corners=3pt,pattern=grid] (0.075*\l,-0.15*\l) rectangle (0.5*\l,-0.3*\l);
        \draw [rounded corners=3pt,pattern=grid] (0.5*\l,-0.15*\l) rectangle (0.65*\l,-0.3*\l);
        \draw [rounded corners=3pt,pattern=grid] (0.65*\l,-0.15*\l) rectangle (\l,-0.3*\l);
        % -2
        \draw [rounded corners=3pt] (-\l,0.5*\l) rectangle (-0.075*\l,0.65*\l);
        \draw [rounded corners=3pt] (0.5*\l,-0.15*\l) rectangle (0.65*\l,\l);
        \draw [rounded corners=3pt] (-\l,-0.55*\l) rectangle (-0.075*\l,-0.7*\l);
        \draw [rounded corners=3pt] (0.075*\l,-0.55*\l) rectangle (\l,-0.7*\l);
    \end{tikzpicture}} \\
    % THIRD ROW
    \subfloat[$\ell=1$ separator clustering hierarchy]{\begin{tikzpicture}[every node/.style={draw,circle}]
        \node[fill=black!75] (a) at (0 ,0) {};
        \node[fill=black!50] (b) at (-1.5, -0.5) {};
        \node[fill=black!50] (c) at (-0.5, -0.5) {};
        \node[fill=black!50] (d) at (0, -0.5) {};
        \node[fill=black!50] (e) at (0.5, -0.5) {};
        \node[fill=black!50] (f) at (1.5, -0.5) {};
        \node[fill=black!20] (g) at (-2, -1) {};
        \node[fill=black!20] (h) at (-1.5, -1) {};
        \node[fill=black!20] (i) at (-1, -1) {};
        \node[fill=black!20] (j) at (-0.5, -1) {};
        \node[fill=black!20] (k) at (0, -1) {};
        \node[fill=black!20] (l) at (0.5, -1) {};
        \node[fill=black!20] (m) at (1, -1) {};
        \node[fill=black!20] (n) at (1.5, -1) {};
        \node[fill=black!20] (o) at (2, -1) {};
        \draw (a) edge (b);
        \draw (a) edge (c);
        \draw (a) edge (d);
        \draw (a) edge (e);
        \draw (a) edge (f);
        \draw (b) edge (g);
        \draw (b) edge (h);
        \draw (b) edge (i);
        \draw (c) edge (j);
        \draw (d) edge (k);
        \draw (e) edge (l);
        \draw (f) edge (m);
        \draw (f) edge (n);
        \draw (f) edge (o);
    \end{tikzpicture}} \;
    \subfloat[$\ell=2$ separators clustering hierarchy]{
    \begin{tikzpicture}
        \draw [draw=white] (0,0) rectangle (0.3,1);
    \end{tikzpicture}
    \begin{tikzpicture}[every node/.style={draw,circle,fill=black!50}]
        \node[pattern=north west lines] (a) at (0 ,0) {};
        \node[pattern=grid] (b) at (0,-0.5) {};
        \draw (a) edge (b);
    \end{tikzpicture} \;
    \begin{tikzpicture}[every node/.style={draw,circle,fill=black!50}]
        \node[pattern=north west lines] (a) at (0 ,0) {};
        \node[pattern=grid] (b) at (-1,-0.5) {};
        \node[pattern=grid] (c) at (0,-0.5) {};
        \node[pattern=grid] (d) at (1,-0.5) {};
        \draw (a) edge (b);
        \draw (a) edge (c);
        \draw (a) edge (d);
    \end{tikzpicture}
    \begin{tikzpicture}
        \draw [draw=white] (0,0) rectangle (0.3,1);
    \end{tikzpicture}
    } \quad
    \subfloat[$\ell=3$ separators clustering hierarchy]{
    \begin{tikzpicture}
        \draw [draw=white] (0,0) rectangle (0.5,1);
    \end{tikzpicture}
    \begin{tikzpicture}[every node/.style={draw,circle}]
        \node at (0,0) {};
    \end{tikzpicture}\;
    \begin{tikzpicture}[every node/.style={draw,circle}]
        \node at (0,0) {};
    \end{tikzpicture}\;
    \begin{tikzpicture}[every node/.style={draw,circle}]
        \node at (0,0) {};
    \end{tikzpicture}\;
    \begin{tikzpicture}[every node/.style={draw,circle}]
        \node at (0,0) {};
    \end{tikzpicture}
    \begin{tikzpicture}
        \draw [draw=white] (0,0) rectangle (0.5,1);
    \end{tikzpicture}
    }
    \caption{A modified ND ordering \& clustering (\autoref{algo:ordering_clustering}). The top row indicates the separators computed at each 
step by separating interiors \& boundaries. The middle row illustrates the clustering of dofs in each separator creating 
the interfaces between interiors. The bottom row shows the clusters hierarchy within each ND separator.}
    \label{fig:ordering_clustering_process}
\end{figure}
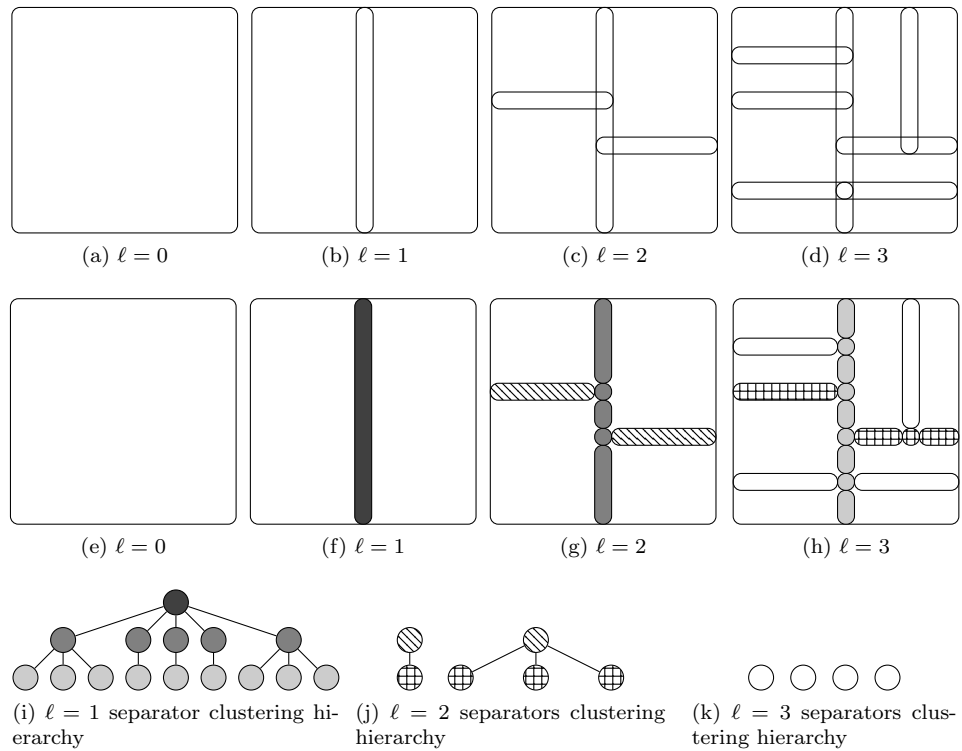

\rev{In practice (see \autoref{sec:numerical_experiments}), we implement this algorithm in two ways. 
If geometry information is available, the \texttt{vertex-separator} subroutine of \autoref{algo:ordering_clustering} is implemented using a recursive coordinate bisection. 
The subgraph is partitioned into two equal parts along the dimension with the largest span, and the nodes in the first part adjacent to the second form the middle separator.
If no geometry information is available, we use the \texttt{nodeND} routine of Metis \cite{karypis1998fast}.}

\subsection{Separators Elimination using Block Cholesky} \label{sec:elimination}

Now that the matrix has been ordered and that dofs have been grouped into clusters defining various interfaces, the next step is to eliminate
the interiors or separators at a given level $\ell$ of the ND tree, as in a usual direct solver (see \autoref{algo:high_level_algo}).
This section describes this elimination step, which is simply a standard block Cholesky reinterpreted with our notation.
Consider $A$ into the ``block-arrowhead'' form following the ND ordering
\[ A = \begin{bmatrix} A_{ss} &        & A_{sn} \\
                              & A_{ww} & A_{wn} \\
                       A_{ns} & A_{nw} & A_{nn} \end{bmatrix} \]
We indicate the separator or interior of interest by $s$, its neighbors by $n$ 
and all disconnected nodes by $w$. By symmetry, $A_{ab} = A_{ba}^\top$.

Let $L_s L_s^\top = A_{ss}$ the Cholesky factorization of $A_{ss}$. Then, define
\[ E_s = 
    \begin{bmatrix} L_s^{-1}             &   & \\
                                         & I & \\
                    - A_{ns} A_{ss}^{-1} & & I \end{bmatrix}
\] 
Then, applying $E_s$ on the left and right of $A$ leads to
\[ E_s A E_s^\top =
    \begin{bmatrix} I & & \\
                      & A_{ww} & A_{wn} \\
                      & A_{nw} & A_{nn} - A_{ns} A_{ss}^{-1} A_{sn} \end{bmatrix}
    =
\begin{bmatrix} I & & \\ & A_{ww} & A_{wn} \\ & A_{nw} & B_{nn} \end{bmatrix}
\]
We notice that this may introduces (potentially many) new $n_i$--$n_j$ edges not present before, a fill-in.
However, there was no modification involving $w$. This is key in the ND ordering: there are no edges $s$--$w$, so no fill-in 
outside the neighbors.

% \autoref{fig:elimination_graph} shows the elimination process from the matrix' graph perspective.
% We see that $s$ is now isolated from the rest of the graph. We say that $s$ has been \emph{eliminated}.
% Furthermore, the effect on the \emph{rest} of the graph was to update the self edge $n$--$n$, $A_{nn}$.
% Separated nodes ($w$) remain untouched.

\subsection{Interfaces Scaling} \label{sec:scaling}
Once that the separators or interiors at a given level have been eliminated, the algorithm goes through each interface and sparsifies it.
However, a critical step before this is the proper scaling of each of those clusters. The goal is to scale (what is left of) $A$
such that each diagonal block corresponding to a given interface is the identity. 
This provides theoretical guarantees (\autoref{sec:theoretical_results}) and significantly improves the accuracy of
the preconditioner (\autoref{sec:numerical_experiments}).

Consider the matrix
\[ A = \begin{bmatrix} 
        A_{pp}  & A_{pn} \\
        A_{np}  & A_{nn} 
        \end{bmatrix} 
\]
We define the block-scaling operation over $p$ as
\[ S_p = 
    \begin{bmatrix} 
        L_p^{-1}    & \\
                    & I 
    \end{bmatrix}
\] 
The result is
\[ S_p A S_p^\top = 
    \begin{bmatrix}
                I       & L_p^{-1} A_{pn}   \\
    A_{np} L_s^{-\top}  & A_{nn} 
    \end{bmatrix}
                = 
    \begin{bmatrix}
                I       & C_{pn} \\
                C_{np}  & A_{nn} 
    \end{bmatrix} 
\]

\subsection{Interface Sparsification using Low-Rank Approximations} \label{sec:sparsification}

Now that interiors have been eliminated and each interface scaled, the final step is the sparsification.
At this stage, the algorithm will go through each interface, $p$, and sparsify it, using low-rank approximations. Consider again
\[ A =
    \begin{bmatrix}
        A_{pp}  & A_{pn} \\
        A_{np}  & A_{nn}
    \end{bmatrix}
\]

\subsubsection{Using orthogonal transformations} \label{sec:sparsification_orthogonal}

Let us assume $A_{pp} = I$. This is not a loss of generality, as it can always be obtained by scaling $p$, as described in the previous section.
Let us also assume that $A_{pn}$ can be well approximated by a low-rank matrix, i.e.,
\[ A_{pn} = Q_{pc} W_{cn} + Q_{pf} W_{fn},\quad \|W_{fn}\|_2 =  \OO{\varepsilon} \]
where $Q_{pc}$ is a thin orthogonal matrix and $Q_{pf}$ its complement. 
This can be computed using a rank-revealing QR (RRQR) 
or a singular value decomposition (SVD) \cite{golub2013matrix, chan1987rank, gu1996efficient}. 
We use the letters $c$ to denote the ``coarse'' (also known as ``skeleton'' or ``relevant'', \cite{ho2016hierarchical}) 
dofs, and $f$ the ``fine'' (``redundant'' or ``irrelevant'') dofs.
Let $Q_{pp}$ be a square orthogonal matrix built as $Q_{pp} = \begin{bmatrix} Q_{pf} & Q_{pc}\end{bmatrix}$. This implies
\[ Q_{pc}^\top A_{pn} = W_{cn},\quad Q_{pf}^\top A_{pn} = W_{fn} =\OO{\varepsilon} \]
Then, define
\begin{equation} \label{eq:orthogonal}
   Q_p = 
    \begin{bmatrix} Q_{pp}  & \\
                            & I 
    \end{bmatrix} 
\end{equation}
We see that
\[ Q_p^\top A Q_p = 
    \begin{bmatrix} 
            I           &               & W_{fn}       \\ 
                        & I             & W_{cn}       \\
            W_{fn}^\top & W_{cn}^\top   & A_{nn} 
    \end{bmatrix}
=   \begin{bmatrix}
            I                   &               & \OO{\varepsilon}       \\ 
                                & I             & W_{cn}       \\
            \OO{\varepsilon}    & W_{cn}^\top   & A_{nn} 
    \end{bmatrix} 
\]

After the orthogonal transformation, $f$ only has very ``weak'' connections to $n$.
If we ignore the $\OO{\varepsilon}$ term, this is the same as dropping the $n$--$f$ edge.
This effectively means $f$ has been eliminated.

However, note that this did \emph{not} introduce any new edge with any of the neighbors of $p$. 
This is the key difference with a ``regular'' elimination as described previously: we can eliminate part of a cluster, 
here $f$, \emph{without forming new edges between its neighbors}. The $n$--$n$ edge is unaffected by this operation (i.e., there is
no fill-in). 
A regular elimination, on the other hand, would have changed the edges $n$--$n$.

\subsubsection{Variant using Interpolative Transformations}
\label{sec:sparsification_interpolative}
The previous section details the sparsification process using orthogonal transformations.
However, this can also be done using other transformations. In this section we explain one variant using interpolative
factorization, which was the original idea in \cite{ho2016hierarchical}.

Assume we can \emph{partition} $p = c \cup f$ (so in this case $c$ and $f$ are subsets of $p$) such that
\[ A_{n f} = A_{n c} T_{cf} + \OO{\varepsilon}. \]
This is often called ``interpolative decomposition''. It can be computed for instance using a rank-revealing QR (RRQR) 
factorization \cite{cheng2005compression} (note that the RRQR is computed over $A_{np}$ instead of $A_{pn}$ in 
\autoref{sec:sparsification_orthogonal}): computing a RRQR over $A_{np}$ leads to (with $P$ the permutation, and 
$R_{22} = \OO{\varepsilon}$)
\[ \begin{bmatrix} A_{nc} & A_{nf} \end{bmatrix} = A_{np} P = \begin{bmatrix} Q_1 & Q_2 \end{bmatrix} \begin{bmatrix} R_{11} & R_{12} \\ & R_{22} \end{bmatrix}\]
\[ \Rightarrow A_{nf} = \underbrace{Q_1 R_{11}}_{A_{nc}} \underbrace{R_{11}^{-1} R_{12}}_{T_{cf}} + \underbrace{Q_2 R_{22}}_{=\OO{\varepsilon}} \]
Note that this factorization can also be computed using randomized methods \cite{liberty2007randomized}.
This technique is referred to as ``interpolative'' because it is exact on $A_{nc}$: only $A_{nf}$ is approximated and $T_{cf}$ acts as an interpolation operator (i.e., as a set of Lagrange basis functions).

Now, consider
\[ T_{p} = 
    \begin{bmatrix} 
        I           &               & \\
        -T_{cf}     & I             & \\
                    &               & I 
    \end{bmatrix} 
\]
Notice how $T_p$ is a \rev{lower}-triangular matrix, while $Q_p$ in \autoref{eq:orthogonal} was orthogonal.
Both can be efficiently inverted; however, working with orthogonal matrices \rev{brings} stability guarantees 
(see \autoref{sec:theoretical_results}).
Then,
\[ T_p^{\top} A T_p = 
    \begin{bmatrix} 
        C_{f f}             & C_{f c}           & \OO{\varepsilon}               \\
        C_{c f}             & A_{c c}           & A_{c n}                       \\
        \OO{\varepsilon}    & A_{n c}           & A_{nn} 
    \end{bmatrix}  
\]
with 
\[ C_{ff} = A_{ff} - A_{fc}T_{cf} - T_{cf}^\top A_{cf} + T_{cf}^\top A_{cc} T_{cf},\; C_{cf} = A_{cf} - A_{cc} T_{cf},\; C_{fc} = C_{cf}^\top \]

The final result is the same as using orthogonal transformation. The differences are that
\begin{itemize}
    \item $A_{pp}$ is not required to be identity;
    \item $A_{cn}$ is simply a subset of $A_{pn}$.
\end{itemize}
However, as we will see later on, there is a significant accuracy loss when using this technique without block scaling 
as opposed to orthogonal transformations with block scaling. Furthermore, it does not guarantee that the approximation 
stays SPD.

\subsection{Clusters merge} \label{sec:merging}

Finally, once we have eliminated all separators at a given level, we need to merge the interfaces of every remaining ND separator.
Consider for instance \autoref{fig:ordering_clustering_process}. 
After having eliminated the leaf (level $\ell = 4$) and the level $\ell = 3$ separators, we need to merge the clusters in each separator. 
This is done following the cluster trees.
Merging children clusters $p_1, \dots, p_k$ into a parent cluster $p$ simply means concatenating their dofs:
\[ p = \begin{bmatrix} p_1 & p_2 & \dots & p_k \end{bmatrix}. \]
Then, all block rows and columns corresponding to $p_1, \dots, p_k$ get concatenated into $p$.

\subsection{Sparsified Nested Dissection}

Now that we have introduced all the required building blocks (block elimination, scaling and sparsification), we can
present the complete algorithm.
Given a matrix $A$, appropriately ordered and clustered, the algorithm simply consists of applying a sequence of 
eliminations $E_s$ (\autoref{sec:elimination}), scalings $S_p$ (\autoref{sec:scaling}) and sparsification $Q_p$
(\autoref{sec:sparsification}) (plus potentially some re-orderings and permutations to take care of the fine nodes $f$ and the merge), at each level $\ell$, effectively reducing $A$ to (approximately) $I$:
\[ M^\top A M \approx I \text{ with } M = \prod_{\ell=1}^L \left( 
    \prod_{s\in S_\ell} E_s^\top
    \prod_{p \in C_\ell} S_p^\top
    \prod_{p \in C_\ell} Q_p  
    \right) \]
In this expression, $S_\ell$ are all the ND separators at level $\ell$ and $C_\ell$ are all the clusters (interfaces) in the graph
right after level $\ell$ elimination.
Since $M$ is given as a product of elementary transformations, it can easily be inverted.
We refer to this algorithm as \algo{}, which stands for ``sparsified Nested Dissection''.
\autoref{algo:spand} presents the algorithm.

\begin{algorithm}
    \begin{algorithmic}
        \REQUIRE{$A \succ 0$; $L > 0$; $\varepsilon$}
        \STATE $M = []$ (empty list)
        \STATE Compute a $L$-levels modified ND ordering of $A$ using \autoref{algo:ordering_clustering}. Infer clusters hierarchy in each ND separator.
        \FORALL{$\ell = L,\dots,1$}
            % ELIMINATE
            \FORALL[Eliminate separators at level $\ell$]{$s$ separator at level $\ell$} 
                \STATE Eliminate $s$, get $E_s$ (\autoref{sec:elimination})
                \STATE Append $E_s$ to $M$
            \ENDFOR            
            % SCALE
            \FORALL[Scale interface]{$p$ interfaces}
                \STATE Scale $p$, get $S_p$ (\autoref{sec:scaling})
                \STATE Append $S_p$ to $M$
            \ENDFOR
            % SPARSIFY
            \FORALL[Sparsify interfaces]{$p$ interface}
                \STATE Sparsify $p$ with accuracy $\varepsilon$, get $Q_p$ (\autoref{sec:sparsification})
                \STATE Append $Q_p$ to $M$
            \ENDFOR
            % MERGE
            \FORALL[Merge clusters]{$s$ separator}
                \STATE Merge interfaces of $s$ one level following clusters hierarchy (\autoref{sec:merging})
            \ENDFOR
        \ENDFOR
        \RETURN $M = \prod_{\ell=1}^L \left( \prod_{s\in S_\ell} E_s^\top \prod_{p \in C_\ell} S_p^\top \prod_{p \in C_\ell} Q_p \right)$
         (such that $ M^\top A M \approx I$)
    \end{algorithmic}
    \caption{The \algo{} algorithm (\texttt{OrthS}).}
    \label{algo:spand}
\end{algorithm}

We illustrate the effect of all the $E_s$, $S_p$ and $Q_p^\top$ in $A$ (i.e., the trailing matrix) on 
\autoref{fig:spand-illustration}. The two top rows show the actual trailing matrix, while the two bottom rows
show the evolution of the matrix graph's clusters as the elimination and sparsification proceeds.

\begin{figure}\centering
\subfloat[][$A$]{\begin{tikzpicture}
    \node[inner sep=0pt] at (0,0){\includegraphics[clip,width=0.28\textwidth]{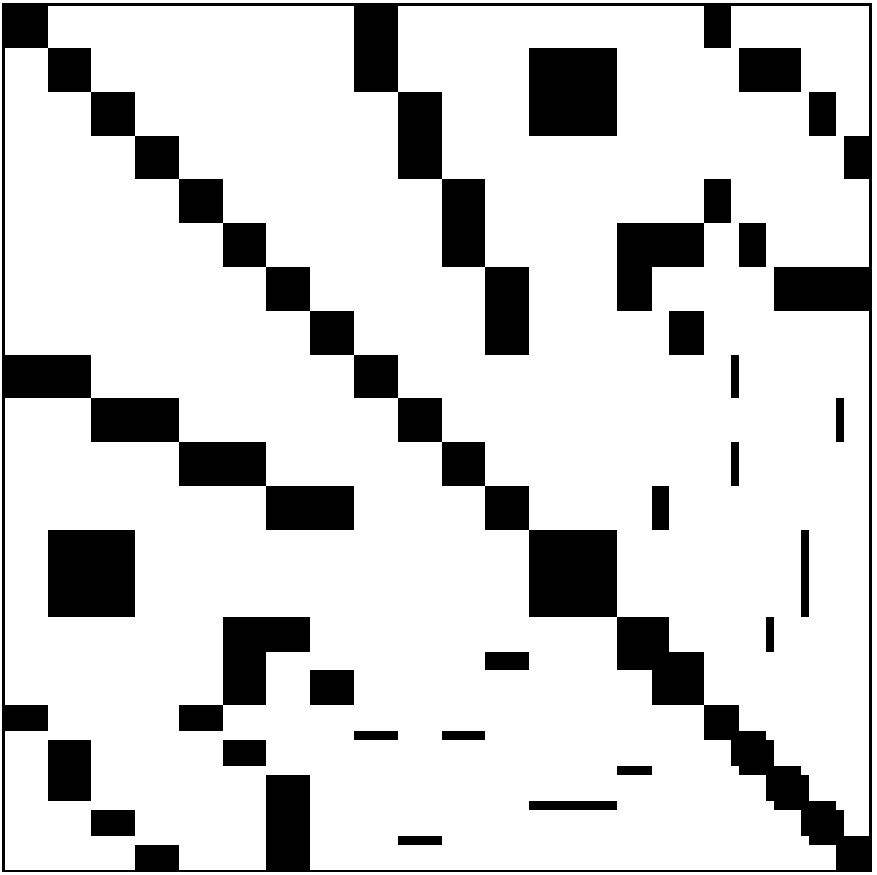}};
    \end{tikzpicture}}\,
    \subfloat[][After $E_1^\top$]{\begin{tikzpicture}
    \node[inner sep=0pt] at (0,0){\includegraphics[clip,width=0.28\textwidth]{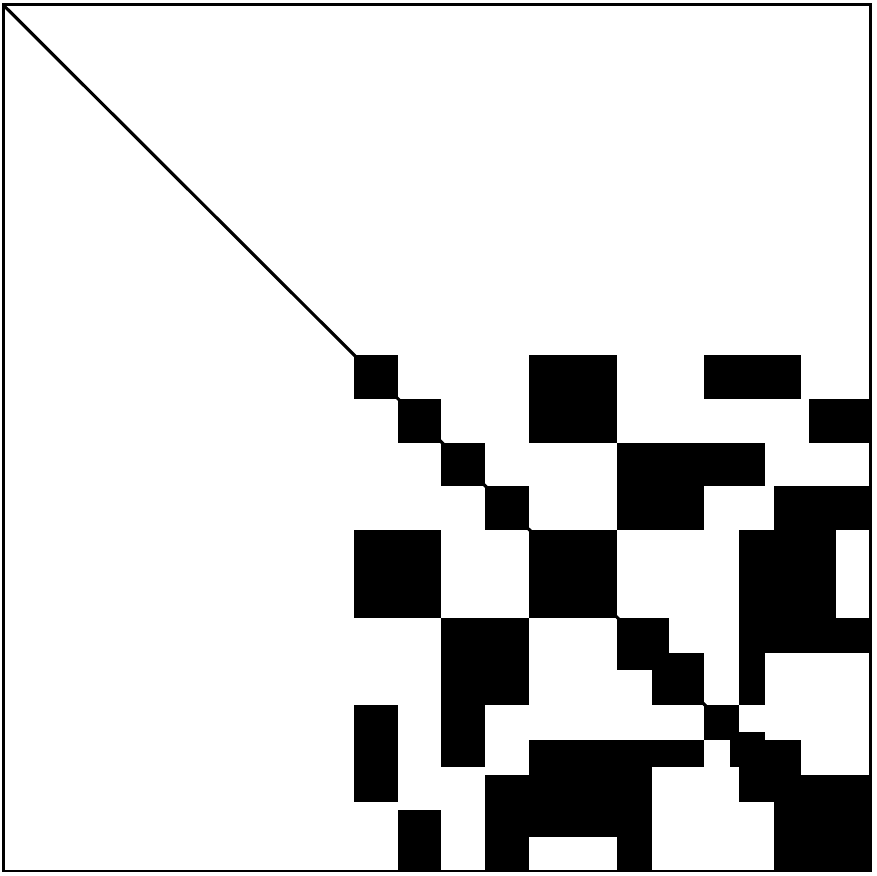}};
    \end{tikzpicture}}\,
    \subfloat[][After $S_1^\top Q_1$. Grey = $\OO{\varepsilon}$.]{\begin{tikzpicture}
    \node[inner sep=0pt] at (0,0){\includegraphics[clip,width=0.28\textwidth]{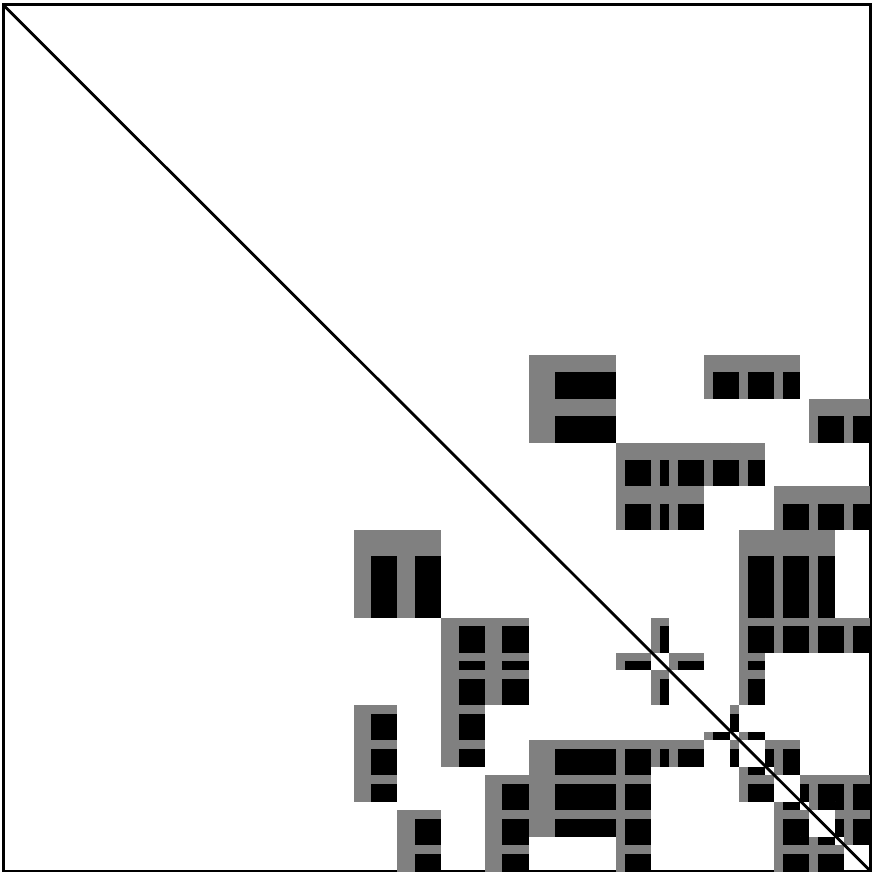}};
    \end{tikzpicture}}\\
    \subfloat[][After reordering, bringing $f$ in front.]{\begin{tikzpicture}
    \node[inner sep=0pt] at (0,0){\includegraphics[clip,width=0.28\textwidth]{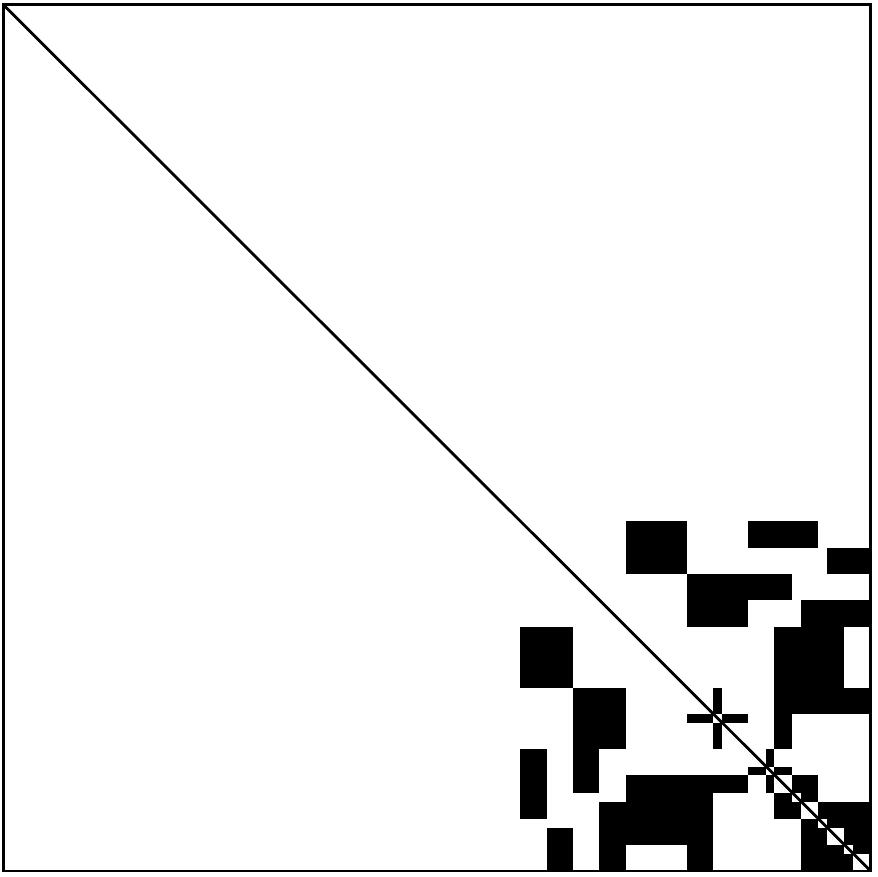}};
    \end{tikzpicture}}\,
    \subfloat[][After \rev{merge} \& $E_2^\top$]{\begin{tikzpicture}
    \node[inner sep=0pt] at (0,0){\includegraphics[clip,width=0.28\textwidth]{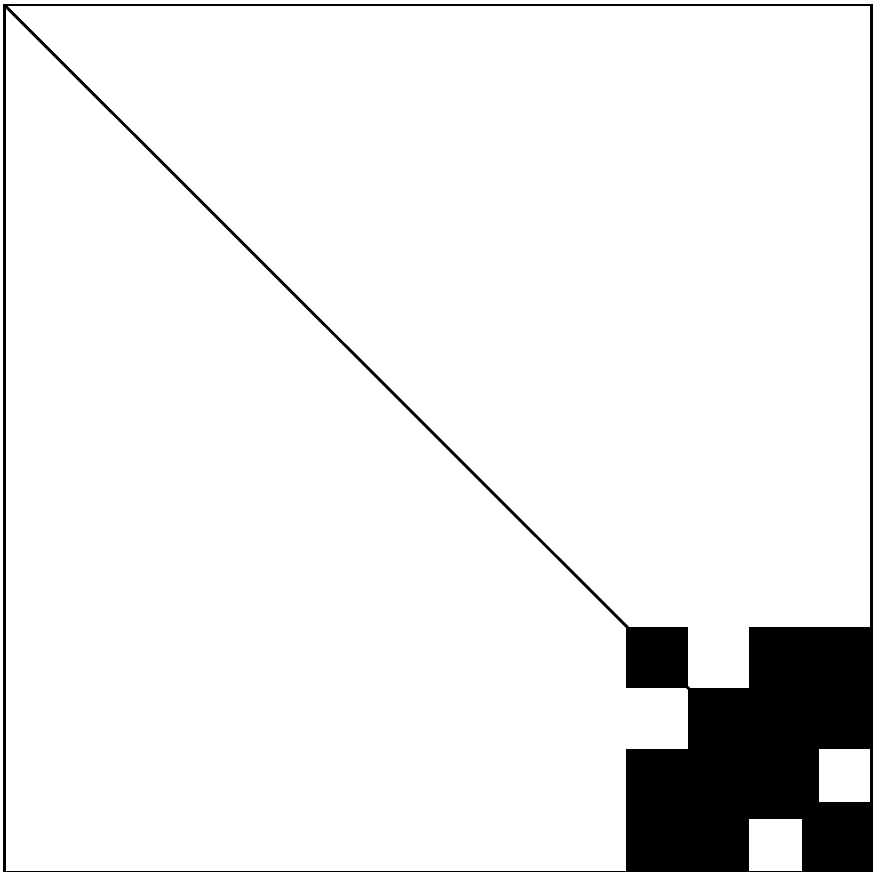}};
    \end{tikzpicture}}\,
    \subfloat[][After $S_2^\top Q_2$. Grey = $\OO{\varepsilon}$]{\begin{tikzpicture}
    \node[inner sep=0pt] at (0,0){\includegraphics[clip,width=0.28\textwidth]{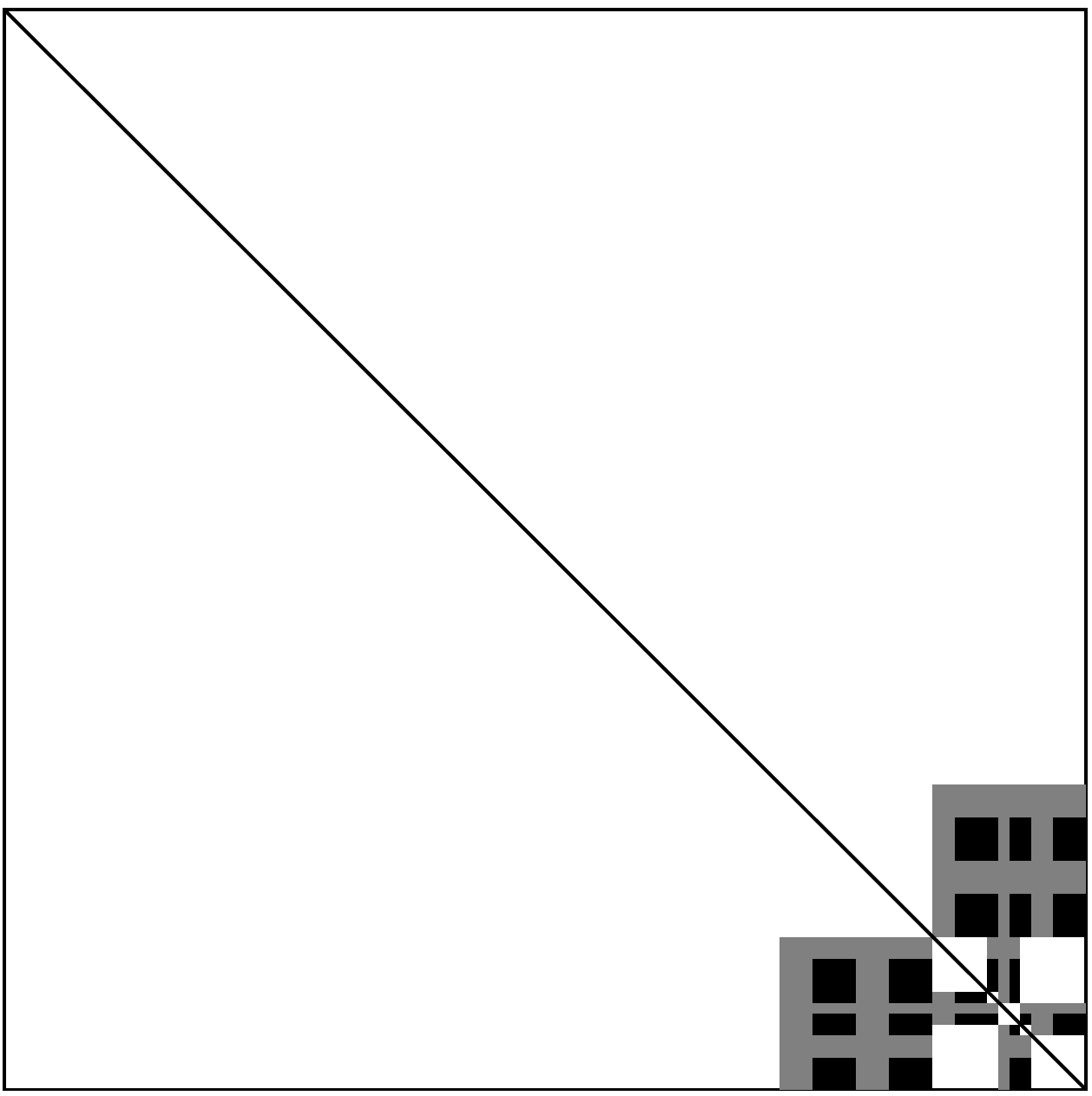}};
    \end{tikzpicture}}\,
 \\
 \subfloat[][\label{fig:spand-illustration-g}$A$, original graph]{\begin{tikzpicture}
        \def \l{1.2};
        % 0 bottom-top
        \draw [rounded corners=3pt,fill=hcgreen] (-0.075*\l,-\l) rectangle (0.075*\l,-0.70*\l);
        \draw [rounded corners=3pt,fill=hcgreen] (-0.075*\l,-0.7*\l) rectangle (0.075*\l,-0.55*\l);
        \draw [rounded corners=3pt,fill=hcgreen] (-0.075*\l,-0.55*\l) rectangle (0.075*\l,-0.3*\l);
        \draw [rounded corners=3pt,fill=hcgreen] (-0.075*\l,-0.3*\l) rectangle (0.075*\l,-0.15*\l);
        \draw [rounded corners=3pt,fill=hcgreen] (-0.075*\l,-0.15*\l) rectangle (0.075*\l,0.1*\l);
        \draw [rounded corners=3pt,fill=hcgreen] (-0.075*\l,0.1*\l) rectangle (0.075*\l,0.25*\l);
        \draw [rounded corners=3pt,fill=hcgreen] (-0.075*\l,0.25*\l) rectangle (0.075*\l,0.5*\l);
        \draw [rounded corners=3pt,fill=hcgreen] (-0.075*\l,0.5*\l) rectangle (0.075*\l,0.65*\l);
        \draw [rounded corners=3pt,fill=hcgreen] (-0.075*\l,0.65*\l) rectangle (0.075*\l,\l);
        % 1
        \draw [rounded corners=3pt,fill=hcred] (-\l,0.1*\l) rectangle (-0.075*\l,0.25*\l);
        \draw [rounded corners=3pt,fill=hcred] (0.075*\l,-0.15*\l) rectangle (0.5*\l,-0.3*\l);
        \draw [rounded corners=3pt,fill=hcred] (0.5*\l,-0.15*\l) rectangle (0.65*\l,-0.3*\l);
        \draw [rounded corners=3pt,fill=hcred] (0.65*\l,-0.15*\l) rectangle (\l,-0.3*\l);
        % 2
        \draw [rounded corners=3pt,fill=hcblue] (-\l,0.5*\l) rectangle (-0.075*\l,0.65*\l);
        \draw [rounded corners=3pt,fill=hcblue] (0.5*\l,-0.15*\l) rectangle (0.65*\l,\l);
        \draw [rounded corners=3pt,fill=hcblue] (-\l,-0.55*\l) rectangle (-0.075*\l,-0.7*\l);
        \draw [rounded corners=3pt,fill=hcblue] (0.075*\l,-0.55*\l) rectangle (\l,-0.7*\l);
        % 3
        \draw [rounded corners=3pt] (-\l,-\l)      rectangle (-0.075*\l,-0.7*\l);
        \draw [rounded corners=3pt] (-\l,-0.55*\l) rectangle (-0.075*\l,0.1*\l);
        \draw [rounded corners=3pt] (-\l,0.25*\l)  rectangle (-0.075*\l,0.5*\l);
        \draw [rounded corners=3pt] (-\l,0.65*\l)  rectangle (-0.075*\l,\l);
        \draw [rounded corners=3pt] (0.075*\l,-\l) rectangle (\l,-0.7*\l);
        \draw [rounded corners=3pt] (0.075*\l,-0.55*\l) rectangle (\l,-0.3*\l);
        \draw [rounded corners=3pt] (0.075*\l,-0.15*\l) rectangle (0.5*\l,\l);
        \draw [rounded corners=3pt] (0.65*\l,-0.15*\l) rectangle (\l,\l);
    \end{tikzpicture}}\noindent\;
    \subfloat[][After $E_1^\top$]{\begin{tikzpicture}
        \def \l{1.2};
        \draw [white] (-\l,-\l) rectangle (\l,\l);
        % 0
        \draw [rounded corners=3pt,fill=hcgreen] (-0.075*\l,-\l) rectangle (0.075*\l,-0.70*\l);
        \draw [rounded corners=3pt,fill=hcgreen] (-0.075*\l,-0.7*\l) rectangle (0.075*\l,-0.55*\l);
        \draw [rounded corners=3pt,fill=hcgreen] (-0.075*\l,-0.55*\l) rectangle (0.075*\l,-0.3*\l);
        \draw [rounded corners=3pt,fill=hcgreen] (-0.075*\l,-0.3*\l) rectangle (0.075*\l,-0.15*\l);
        \draw [rounded corners=3pt,fill=hcgreen] (-0.075*\l,-0.15*\l) rectangle (0.075*\l,0.1*\l);
        \draw [rounded corners=3pt,fill=hcgreen] (-0.075*\l,0.1*\l) rectangle (0.075*\l,0.25*\l);
        \draw [rounded corners=3pt,fill=hcgreen] (-0.075*\l,0.25*\l) rectangle (0.075*\l,0.5*\l);
        \draw [rounded corners=3pt,fill=hcgreen] (-0.075*\l,0.5*\l) rectangle (0.075*\l,0.65*\l);
        \draw [rounded corners=3pt,fill=hcgreen] (-0.075*\l,0.65*\l) rectangle (0.075*\l,\l);
        % 1
        \draw [rounded corners=3pt,fill=hcred] (-\l,0.1*\l) rectangle (-0.075*\l,0.25*\l);
        \draw [rounded corners=3pt,fill=hcred] (0.075*\l,-0.15*\l) rectangle (0.5*\l,-0.3*\l);
        \draw [rounded corners=3pt,fill=hcred] (0.5*\l,-0.15*\l) rectangle (0.65*\l,-0.3*\l);
        \draw [rounded corners=3pt,fill=hcred] (0.65*\l,-0.15*\l) rectangle (\l,-0.3*\l);
        % 2
        \draw [rounded corners=3pt,fill=hcblue] (-\l,0.5*\l) rectangle (-0.075*\l,0.65*\l);
        \draw [rounded corners=3pt,fill=hcblue] (0.5*\l,-0.15*\l) rectangle (0.65*\l,\l);
        \draw [rounded corners=3pt,fill=hcblue] (-\l,-0.55*\l) rectangle (-0.075*\l,-0.7*\l);
        \draw [rounded corners=3pt,fill=hcblue] (0.075*\l,-0.55*\l) rectangle (\l,-0.7*\l);
    \end{tikzpicture}}\noindent\;
    \subfloat[][After $S_1^\top Q_1$]{\begin{tikzpicture}
        \def \l{1.2};
        \draw [white] (-\l,-\l) rectangle (\l,\l);
        % 0
        \draw [rounded corners=3pt,fill=hcgreen] (-0.075*\l,-0.95*\l) rectangle (0.075*\l,-0.75*\l);
        \draw [rounded corners=3pt,fill=hcgreen] (-0.075*\l,-0.7*\l)  rectangle (0.075*\l,-0.55*\l);
        \draw [rounded corners=3pt,fill=hcgreen] (-0.075*\l,-0.5*\l)  rectangle (0.075*\l,-0.35*\l);
        \draw [rounded corners=3pt,fill=hcgreen] (-0.075*\l,-0.3*\l)  rectangle (0.075*\l,-0.15*\l);
        \draw [rounded corners=3pt,fill=hcgreen] (-0.075*\l,-0.1*\l)  rectangle (0.075*\l,0.05*\l);
        \draw [rounded corners=3pt,fill=hcgreen] (-0.075*\l,0.1*\l)   rectangle (0.075*\l,0.25*\l);
        \draw [rounded corners=3pt,fill=hcgreen] (-0.075*\l,0.30*\l)  rectangle (0.075*\l,0.45*\l);
        \draw [rounded corners=3pt,fill=hcgreen] (-0.075*\l,0.5*\l)   rectangle (0.075*\l,0.65*\l);
        \draw [rounded corners=3pt,fill=hcgreen] (-0.075*\l,0.7*\l)   rectangle (0.075*\l,0.95*\l);
        % 1
        \draw [rounded corners=3pt,fill=hcred] (-0.9*\l, 0.1*\l)    rectangle (-0.15*\l,0.25*\l);
        \draw [rounded corners=3pt,fill=hcred] (0.15*\l,-0.15*\l)   rectangle (0.4*\l,-0.3*\l);
        \draw [rounded corners=3pt,fill=hcred] (0.5*\l, -0.15*\l)   rectangle (0.65*\l,-0.3*\l);
        \draw [rounded corners=3pt,fill=hcred] (0.7*\l, -0.15*\l)   rectangle (\l,-0.3*\l);
        % 2
        \draw [rounded corners=3pt,fill=hcblue] (-0.9*\l,0.5*\l) rectangle (-0.15*\l,0.65*\l);
        \draw [rounded corners=3pt,fill=hcblue] (0.5*\l,-0.05*\l) rectangle (0.65*\l,0.9*\l);
        \draw [rounded corners=3pt,fill=hcblue] (-0.9*\l,-0.55*\l) rectangle (-0.15*\l,-0.7*\l);
        \draw [rounded corners=3pt,fill=hcblue] (0.15*\l,-0.55*\l) rectangle (0.9*\l,-0.7*\l);
    \end{tikzpicture}}\noindent\;
    \subfloat[][After merge]{\begin{tikzpicture}
        \def \l{1.2};
        \draw [white] (-\l,-\l) rectangle (\l,\l);
        % 0
        \draw [rounded corners=3pt,fill=hcgreen] (-0.075*\l,-0.95*\l) rectangle (0.075*\l,-0.40*\l);
        \draw [rounded corners=3pt,fill=hcgreen] (-0.075*\l,-0.3*\l)  rectangle (0.075*\l,-0.15*\l);
        \draw [rounded corners=3pt,fill=hcgreen] (-0.075*\l,-0.1*\l)  rectangle (0.075*\l,0.05*\l);
        \draw [rounded corners=3pt,fill=hcgreen] (-0.075*\l,0.1*\l)   rectangle (0.075*\l,0.25*\l);
        \draw [rounded corners=3pt,fill=hcgreen] (-0.075*\l,0.35*\l)   rectangle (0.075*\l,0.95*\l);
        % 1
        \draw [rounded corners=3pt,fill=hcred] (-0.9*\l, 0.1*\l)    rectangle (-0.15*\l,0.25*\l);
        \draw [rounded corners=3pt,fill=hcred] (0.15*\l, -0.15*\l)  rectangle (0.9*\l,-0.3*\l);
        % 2
        \draw [rounded corners=3pt,fill=hcblue] (-0.9*\l,0.5*\l) rectangle (-0.15*\l,0.65*\l);
        \draw [rounded corners=3pt,fill=hcblue] (0.5*\l,-0.05*\l) rectangle (0.65*\l,0.9*\l);
        \draw [rounded corners=3pt,fill=hcblue] (-0.9*\l,-0.55*\l) rectangle (-0.15*\l,-0.7*\l);
        \draw [rounded corners=3pt,fill=hcblue] (0.15*\l,-0.55*\l) rectangle (0.9*\l,-0.7*\l);
    \end{tikzpicture}}\\ \noindent\;
    \subfloat[][After $E_2^\top$]{\begin{tikzpicture}
        \def \l{1.2};
        \draw [white] (-\l,-\l) rectangle (\l,\l);
        % 0
        \draw [rounded corners=3pt,fill=hcgreen] (-0.075*\l,-0.95*\l) rectangle (0.075*\l,-0.40*\l);
        \draw [rounded corners=3pt,fill=hcgreen] (-0.075*\l,-0.3*\l)  rectangle (0.075*\l,-0.15*\l);
        \draw [rounded corners=3pt,fill=hcgreen] (-0.075*\l,-0.1*\l)  rectangle (0.075*\l,0.05*\l);
        \draw [rounded corners=3pt,fill=hcgreen] (-0.075*\l,0.1*\l)   rectangle (0.075*\l,0.25*\l);
        \draw [rounded corners=3pt,fill=hcgreen] (-0.075*\l,0.35*\l)   rectangle (0.075*\l,0.95*\l);
        % 1
        \draw [rounded corners=3pt,fill=hcred] (-0.9*\l, 0.1*\l)    rectangle (-0.15*\l,0.25*\l);
        \draw [rounded corners=3pt,fill=hcred] (0.15*\l, -0.15*\l)  rectangle (0.9*\l,-0.3*\l);
    \end{tikzpicture}}\noindent\;
    \subfloat[][After $S_2^\top Q_2$]{\begin{tikzpicture}
        \def \l{1.2};
        \draw [white] (-\l,-\l) rectangle (\l,\l);
        % 0
        \draw [rounded corners=3pt,fill=hcgreen] (-0.075*\l,-0.85*\l) rectangle (0.075*\l,-0.5*\l);
        \draw [rounded corners=3pt,fill=hcgreen] (-0.075*\l,-0.3*\l)  rectangle (0.075*\l,-0.15*\l);
        \draw [rounded corners=3pt,fill=hcgreen] (-0.075*\l,-0.1*\l)  rectangle (0.075*\l,0.05*\l);
        \draw [rounded corners=3pt,fill=hcgreen] (-0.075*\l,0.1*\l)   rectangle (0.075*\l,0.25*\l);
        \draw [rounded corners=3pt,fill=hcgreen] (-0.075*\l,0.45*\l)   rectangle (0.075*\l,0.85*\l);
        % 1
        \draw [rounded corners=3pt,fill=hcred] (-0.8*\l, 0.1*\l)    rectangle (-0.25*\l,0.25*\l);
        \draw [rounded corners=3pt,fill=hcred] (0.25*\l, -0.15*\l)   rectangle (0.8*\l,-0.3*\l);
    \end{tikzpicture}}\noindent\;
    \subfloat[][After merge]{\begin{tikzpicture}
        \def \l{1.2};
        \draw [white] (-\l,-\l) rectangle (\l,\l);
        % 0
        \draw [rounded corners=3pt,fill=hcgreen] (-0.075*\l,-0.5*\l) rectangle (0.075*\l,0.5*\l);
        % 1
        \draw [rounded corners=3pt,fill=hcred] (-0.8*\l, 0.1*\l)    rectangle (-0.25*\l,0.25*\l);
        \draw [rounded corners=3pt,fill=hcred] (0.25*\l, -0.15*\l)  rectangle (0.8*\l,-0.3*\l);
    \end{tikzpicture}}\noindent\;
    \subfloat[][After $E_3^\top$]{\begin{tikzpicture}
        \def \l{1.2};
        \draw [white] (-\l,-\l) rectangle (\l,\l);
        % 0
        \draw [rounded corners=3pt,fill=hcgreen] (-0.075*\l,-0.5*\l) rectangle (0.075*\l,0.5*\l);
    \end{tikzpicture}}\noindent\;
\caption{Illustration of the spaND algorithm.
Given $A$, create a ND tree of depth 4 and cluster $A$ accordingly, as shown on \autoref{fig:spand-illustration-g}.
This cartoon shows clusters of vertices of $A$, where the edges (not shown) should be thought as connecting close 
neighbors (like on a regular 2D grid).
Denote by $E_\ell$, $S_\ell$ and $Q_\ell$ all eliminations, scalings and sparsifications at level $\ell$.
Then, we have 
$E_4 E_3 ( Q_2^\top S_2 E_2 ) ( Q_1^\top S_1 E_1 ) A ( E_1^\top S_1^\top Q_1 ) ( E_2^\top S_2^\top Q_2 ) E_3^\top E_4^\top \approx I$.
The top rows show the evolution of the trailing matrix; bottom rows show the evolution of the matrix graph after eliminations,
sparsifications and merges.
We represent the sparsification process by shrinking the size of the clusters.}
\label{fig:spand-illustration}
\end{figure}
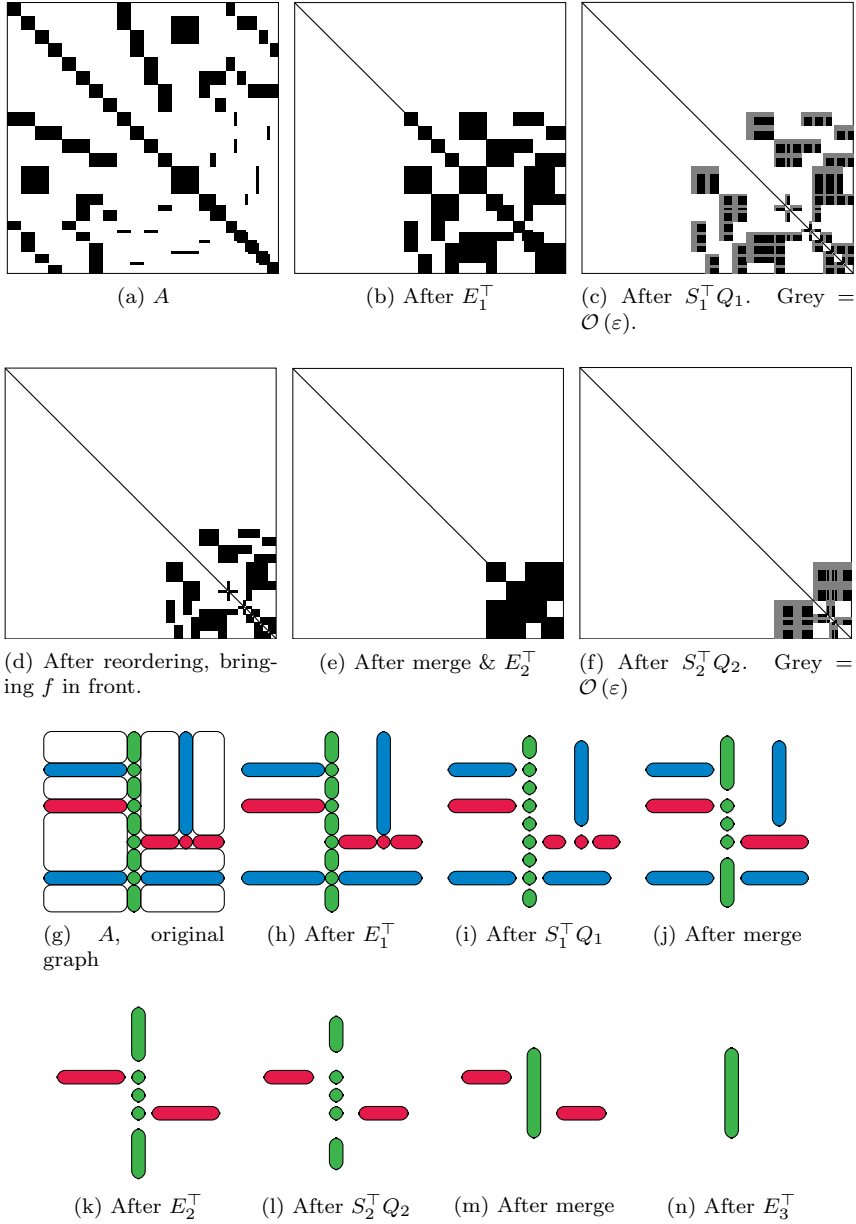

\section{Theoretical results} \label{sec:theoretical_results}

We here discuss a couple of facts related to the above factorizations.

\subsection{Sparsification and Error on the Schur Complement}

Consider a framework where
\[ A = \begin{bmatrix} A_{pp} & A_{pn} \\ A_{np} & A_{nn} \end{bmatrix}. \]
Without loss of generality, we do not include the $w$--$w$ and $w$--$n$ blocks, as they are completely
disconnected from $p$ and unaffected by the sparsification.
Then, consider a general low-rank approximation
\[ A_{pn} = X_1 Y_1^\top + X_2 Y_2^\top \]
where $\| Y_2 \| = \OO{\epsilon}$.
Using $X = \begin{bmatrix} X_1 & X_2 \end{bmatrix}$ as a change of variable, $A$ becomes
\[ \begin{bmatrix} 
        X^{-1}  & \\ 
                & I 
    \end{bmatrix}
    \begin{bmatrix} 
        A_{pp}  & A_{pn}    \\ 
        A_{np}  & A_{nn} 
    \end{bmatrix}
    \begin{bmatrix} 
        X^{-\top}   &   \\ 
                    & I 
    \end{bmatrix} 
    = 
    \begin{bmatrix} 
        B_{pp}  & Y^\top \\ 
        Y       & A_{nn} 
    \end{bmatrix} 
    =
    \begin{bmatrix} 
        B_{11}  & B_{12} & Y_1^\top \\
        B_{21}  & B_{22} & Y_2^\top \\
        Y_1     & Y_2    & A_{nn} 
    \end{bmatrix} 
\]

The sparsification process then assumes $Y_2 = 0$ and eliminates the $2$--$2$ block. 
The true $n$--$n$ Schur complement is
\[ S = A_{nn} - Y_2 B_{22}^{-1} Y_2^\top \]
while the approximate one, ignoring $Y_2$, is simply $A_{nn}$. The error is then
\[ E_{nn} = Y_2 B_{22}^{-1} Y_2^\top. \]

We can now consider the different variants proposed in \autoref{sec:sparsification}:
\begin{itemize}
    \item (\texttt{In}) \algo{} using interpolative factorization and no diagonal block scaling.
    This gives $B_{22} = C_{ff}$ and $Y_2 = A_{nf} - A_{nc} T_{cf}$, so that
    \[ \| E_{nn} \|_2 \leq \| Y_2 \|_2^2 \| B_{22}^{-1} \|_2 = \OO{\varepsilon^2} \| C_{ff}^{-1} \|. \]
    \item (\texttt{InS}) \algo{} using interpolative factorization and diagonal block scaling. This leads to
    \[ \| E_{nn} \|_2 \leq \| Y_2 \|_2^2 \| B_{22}^{-1} \|_2 = \OO{\varepsilon^2} \| C_{ff}^{-1} \|. \]
    However, since $A_{ss} = I$, $C_{ff} = I + T_{cf}^\top T_{cf}$, we can expect, if $T_{cf}$ is small (which happens if the right algorithm is employed, see \cite{miranian2003strong}), 
    $\|C_{ff}^{-1}\|$ to be much closer to $1$.
    \item (\texttt{OrthS}) \algo{} using orthogonal factorization and diagonal block scaling. In this case, we 
    simply have $B_{22} = I$ and $Y_2 = W_{fn}^\top$, and so,
    \[ \| E_{nn} \|_2 \leq \| Y_2 \|_2^2 = \OO{\varepsilon^2}. \]
\end{itemize}
Table \ref{tab:error_approx} summarizes the results. We notice that those three variants have roughly the same cost, since
they \rev{require} a RRQR over $A_{pn}$ or $A_{np}$, and their cost is proportional to $\OO{|p||n||c|}$ with $|c|$ the 
resulting rank \cite[Algorithm 5.4.1]{golub2013matrix}

\begin{table}
    \centering
    {\setlength{\extrarowheight}{5pt}
        \begin{tabular}{l|lll}
        Version             & Error on $n$--$n$                         & Cost &        \\
        \hline
        \texttt{In}         & $\OO{\varepsilon^2} \| C_{ff}^{-1} \|_2$  & $\OO{|p||n||c|}$ &    $C_{ff}$ arbitrary                \\
        \texttt{InS}        & $\OO{\varepsilon^2} \| C_{ff}^{-1} \|_2$  & $\OO{|p||n||c|}$ &    $C_{ff} = I + T_{cf}^\top T_{cf}$ \\
        \texttt{OrthS}      & $\OO{\varepsilon^2}$                      & $\OO{|p||n||c|}$ &                                      \\    
        \end{tabular}
    }
    \caption{Error for various approximations. The left column indicate the sparsification variant:\texttt{In} means 
    interpolative and no scaling; \texttt{InS} means interpolative and scaling; \texttt{OrthS} means orthogonal and scaling.}
    \label{tab:error_approx}
\end{table}

The key is that the interpolative error bound (without and to some extent with scaling) includes the potentially large $\|C_{ff}^{-1}\|_2$ term, which is not present with
the \texttt{OrthS} version.
This indicates that we can expect the versions with diagonal scaling to have smaller errors $E_{nn}$. 
This will be verified in \autoref{sec:numerical_experiments}.

\subsection{Stability of the Block Scaling \& Orthogonal Transformations Variant}
In addition to a smaller $n$--$n$ error as explained previously, the \texttt{OrthS} version provides stability guarantees.

\begin{theorem} Let
    \[ A = \begin{bmatrix} 
        I           & A_{pn} \\
        A_{pn}^\top & A_{nn} \end{bmatrix} \]
    be a SPD matrix. For any low-rank approximation
    \[ A_{pn} = Q_{pf} W_{fn} + Q_{pc} W_{cn} \]
    where $Q_p = \begin{bmatrix} Q_{pf} & Q_{pc} \end{bmatrix}$ is a square orthogonal matrix, 
    \[ B_p = 
        \begin{bmatrix} 
            I           & W_{cn} \\
            W_{cn}^\top & A_{nn} 
        \end{bmatrix}  
    \]
    is SPD.
\end{theorem}

\begin{proof}
    The $n-n$ Schur Complement of $B_p$ (when eliminating $c$) is
    \[ S_B = A_{nn} - W_{cn}^\top W_{cn}. \]
    On the other hand, the $n-n$ Schur Complement of $A$ (when eliminating $p$) is
    \[ S_A = A_{nn} - A_{np}^\top A_{pn} = A_{nn} - W_{cn}^\top W_{cn} - W_{fn}^\top W_{fn} \]
    which implies
    \[ S_B = S_A + \rev{W_{fn}^\top W_{fn}}. \]
    Since $A$ is SPD, so is $S_A$, and since $W_{fc}^\top W_{fc} \succeq 0$, we find that $S_B$ is SPD. 
    Since the $c-c$ \rev{block} of $B_p$ is identity, we conclude that $B_p$ is SPD.
\end{proof}

\begin{corollary}
    For any SPD matrix and $\varepsilon \geq 0$, the sparsified matrices of the spaND algorithm using block diagonal scaling and orthogonal low-rank approximations (\texttt{OrthS}) remain SPD.
    In other words, the algorithm never breaks down.
\end{corollary}

Note that the above corollary does \emph{not} depend on the quality of the low-rank approximation, i.e., it works even for $\epsilon = 0$. 
It merely relies on the fact that the truncated error ($Q_{pf} W_{fn}$) is orthogonal to what is retained ($Q_{pc} W_{cn}$) 
and that the scheme is using a ``weak admissibility'' criterion (\emph{all} edges of $p$ are compressed).
Finally, note that the above proof also shows that
\[ S_B = S_A + \OO{\varepsilon^2}, \quad S_B \succeq S_A. \]
This is a classical result in the case of low-rank approximation using weak admissibility (see \cite{xia2010robust, 
xia2017effective} for instance).

\subsection{Complexity analysis} \label{sec:complexity}

\rev{We discuss the complexity of spaND and contrast it with the usual ND algorithm.}

\paragraph{Classical ND}
\rev{Nested dissection leads to a binary tree decomposition of the graph of $A$ (although $n$-ary trees are possible). In the literature, the nested dissection tree is often defined as a tree of separators. Here for convenience, we take a slightly different viewpoint where each node is a subgraph of $G$. Both viewpoints are equivalent. We start with the root node that corresponds to the full graph $G$ of size $N$. We define the children nodes as the subgraphs that are disconnected by the separator. This process is applied recursively to define the entire tree.}

\rev{In our complexity analysis, we are going to assume that all the graphs for sparse matrices satisfy the following nested dissection property. We assume that leaf nodes contain at most $N_0$ nodes, where $N_0 \in \OO{1}$. Consider a node $i$ in the tree, of size $n_i$. Consider the set $D_i$ of all nodes $j$ such that they are descendant of $i$ and they contain at least $n_i / 2$ nodes. We assume that $| D_i | \in \OO{1}$, that is the size of this set is bounded for all $i$ and $N$. This property is satisfied for $\beta$-balanced trees in which all children subgraphs have size $\beta n_i$, for some $0 < \beta < 1$ independent of $i$ and $N$. In that case we have $| D_i | \leq 1 + \log 2 / \log \beta^{-1}$.} 

\rev{We assume that all separators are minimal in the sense that each node in a separator is connected to the two children subgraphs in the nested dissection partitioning (otherwise this node can be moved to one of the subgraphs).}

\rev{Finally, we assume that a subgraph of size $n_i$ is connected to at most $O(n_i^{2/3})$ nodes in $G$.}

\rev{As far as the authors know, all matrices that arise in the discretization of partial differential equations in 3D using a local stencil satisfy this property.}

\rev{Consider now a node $i$ of size $2^{-\ell} N \le n_i < 2^{-\ell+1} N$ (see \autoref{fig:nested_dissection_3D}). The associated separator has size at most
\[ c_\ell \in \OO{ 2^{-2\ell/3} N^{2/3} } \]
Further, the fill-in results in at most $\OO{2^{-2\ell/3} N^{2/3}}$ \revrev{non-zero} entries in each row. The cost of eliminating a separator in that size range is bounded by
\[ e_l \in \OO{ \left(2^{-2\ell/3} N^{2/3}\right)^3 } = \OO{2^{-2\ell} N^2} \]
From our assumption, the number of clusters of size $2^{-\ell} N \le n_i < 2^{-\ell+1} N$ is bounded by $2^{\ell}$. Hence, the overall factorization cost is bounded by
\[ t_\text{ND,fact} \in \OO{\sum_{\ell=0}^L 2^{\ell} e_{\ell}} = \OO{ \sum_{\ell=0}^L 2^{-\ell} N^2 } = \OO{N^2}, \quad
L \in \Theta( \log(N/N_0) ) \]
We recover the usual computational cost of nested dissection for 3D meshes. Most of the computational expense is at the top of the nested dissection tree, with the final separator of size $N^{2/3}$.}

\rev{The complexity of applying the factorization can be derived similarly. Since for each cluster of size $n_i$, its separator has $\OO{2^{-2\ell/3} N^{2/3}}$ fill-in entries in its rows, the related solve cost is $\OO{2^{-4\ell/3} N^{4/3}}$ and the cost of one solve is
\[ t_\text{ND,apply} \in \OO{ \sum_{\ell=0}^L 2^{-\ell/3} N^{4/3} } = \OO{N^{4/3}} \]%
}

\paragraph{spaND}

\rev{On the other hand, assume that the sparsification is able to decrease each separator size before elimination from $c_\ell$ to
\[ s_\ell \in \OO{ 2^{-\ell/3} N^{1/3} } \]
This means that the rank scales roughly like the diameter of the separators. This is also the rank of the off-diagonal blocks for separators in the original matrix $A$. The assumption in some sense is that the rank of far-away fill-ins is $O(1)$. This is comparable with complexity assumptions in the fast multipole method for example.}

\rev{We now discuss a few additional assumptions regarding the construction of the interfaces, in order to guarantee the final $\OO{N \log N}$ cost. Recall that interfaces are used for sparsification and correspond to a multilevel partitioning of the separators. We will say that two nodes $(i,j)$ at the same level in the nested dissection tree are neighbors if there is a node in $G$ that belongs to a separator at this level or above, and that is connected to $i$ and $j$, in the graph $G$. We will assume that each node has only $\OO{1}$ neighbors. Under this assumption, each interface is connected to $\OO{1}$ interfaces at the same level.}

\rev{Considering the computational cost, for all nodes of size $2^{-\ell} N \le n_i < 2^{-\ell+1} N$, the cost can be divided into:
\begin{itemize}
\item eliminating separators. With the same reasoning as previously, and since an interface is connected to $\OO{1}$ interfaces, the cost is bounded by 
\[ \OO{ (2^{-\ell/3} N^{1/3})^3 } = \OO{2^{-\ell} N} \]
\item scaling and sparsifying the remaining interfaces. By construction the size of each interface is in $\OO{ 2^{-\ell/3} N^{1/3} }$. Since sparsification has cost $\OO{m^2 n}$ for a matrix block of size $m \times n$, the cost of sparsifying one interface is bounded similarly by $\OO{2^{-\ell} N}$.
\end{itemize}
Hence, under our assumptions, the overall factorization cost for spaND is
\[ t_\text{spaND,fact} \in \OO { \sum_{\ell=1}^L 2^\ell \; 2^{-\ell} N } = \OO{N \log N} \]
The complexity of applying the factorization can be derived like previously. A direct calculation leads to
\[ t_\text{spaND,apply} \in \OO{ \sum_{\ell=0}^L 2^\ell \left( 2^{-\ell/3} N^{1/3} \right)^2 } 
= \OO{ \sum_{\ell=1}^L 2^{\ell/3} N^{2/3} } = \OO{N} \]%
}

\rev{Finally, notice that in both cases the memory complexity scales like the factorization application.
Section~\ref{sec:scalings_with_problem_size} presents some experimental results regarding separator sizes as a function of $N$.}

\section{Numerical Experiments} \label{sec:numerical_experiments}

This section presents applications of the algorithm on various problems. All matrices are symmetric, real and 
positive-definite.

We use the following notation throughout this section:
\begin{itemize}
    \item $\tf$ is the factorization time (in seconds), \rev{not} including partitioning;
    \item $\tp$ is the partitioning time (in seconds);
    \item $\ts$ is the total time (in seconds) required for CG to reach a relative residual \rev{$\|Ax-b\|_2/\|b\|_2$} of $10^{-12}$. \rev{It is the total time to reach convergence.}
    While this is quite a small \rev{value}, \rev{being} able to reach those tolerances is a good indication of the numerical stability of the 
    algorithm (i.e., that the preconditioner does not prevent CG from converging to a small tolerance);
    \item $\ncg$ is the associated number of CG steps;
    \item $\st$ is the size of the top separator right before elimination;
    \item $\mf$ is the number of non-zero entries in the factorization;
\end{itemize}
On top of this, at some point we compare \algo{} to classical ``exact'' ND (using \algo{} with no compression \& scaling; 
``Direct'') and to a classical ILU(0) \cite{saad2003iterative} (``ILU(0)'').

All tests where run on a machine with \rev{300 GB} of RAM and a \rev{Intel(R) Xeon(R) Gold 5118 CPU at 2.30GHz}. The algorithm is sequential and was written in C++. We use \rev{GCC 8.1.0} and \rev{Intel(R) MKL 2019 for Linux} for the BLAS \& LAPACK operations. When no geometry information is available, we use Metis 5.1 \cite{karypis1998fast} for the vertex-separator routine. We use Ifpack2 \cite{Ifpack2} for ILU(0). Low-rank approximations are performed using LAPACK's geqp3 \cite{anderson1999lapack}. The truncation uses a simple rule, truncating based on the absolute value of the diagonal entries of the $R$ factor. This means that, given R, we select the first $r$ rows, where $\frac{|R_{ii}|}{|R_{11}|} \geq \varepsilon$ for $1 \leq i \leq r$.

\subsection{Impact of Diagonal Scaling \& Orthogonal Transformations}

In this first set of experiments we compare, empirically, the three variants of the algorithm:
\begin{itemize}
    \item (\texttt{In}) \algo{} using interpolative factorization and no diagonal block scaling;
    \item (\texttt{InS}) \algo{} using interpolative factorization and diagonal block scaling;
    \item (\texttt{OrthS}) \algo{} using orthogonal factorization and diagonal block scaling.
\end{itemize}

This should be contrasted with prior work \cite{ho2016hierarchical} where the algorithm was using the interpolative
only variant (with no scaling).

\subsubsection{High contrast $2D$ Laplacians} \label{sec:2D}

We first consider $2D$ elliptic equations
\begin{equation} \label{eq:elliptic} \nabla(a(x) \cdot \nabla u(x)) = f \quad \forall x \in \Omega = [0,1]^2,\; u|_{\partial \Omega} = 0 \end{equation}
where $a(x)$ is a quantized high contrast field with high of $\rho$ and low of $\rho^{-1}$ 
and where \ref{eq:elliptic} is discretized with a 5-points stencil.
This leads to the following discretization
\begin{align*} & ( a_{i-1/2,j} + a_{i+1/2,j} + a_{i,j-1/2} + a_{i,j+1/2} ) u_{ij} \\
& - a_{i-1/2,j} u_{i-1,j}
- a_{i+1/2,j} u_{i+1,j}
- a_{i,j-1/2} u_{i,j-1}
- a_{i,j+1/2} u_{i,j+1} = h^2 f_{ij} \end{align*}
The field $a$ is built in the following way:
\begin{itemize}
\item create a random $(0,1)$ array $\hat a_{ij}$;
\item smooth $\hat a$ by convolving it with a unit-width Gaussian;
\item quantize $\hat a$ 
\[ a_{ij} = \left\{ \begin{matrix} \rho & \text{ if } & \hat a_{ij} \geq 0.5 \\
                           \rho^{-1} & \text{ else} & \end{matrix} \right. \]
\end{itemize}
\autoref{fig:contrast} gives an example of high contrast field for $n=32$ and $n=128$.

\begin{figure}
\centering
\includegraphics[width=0.3\textwidth]{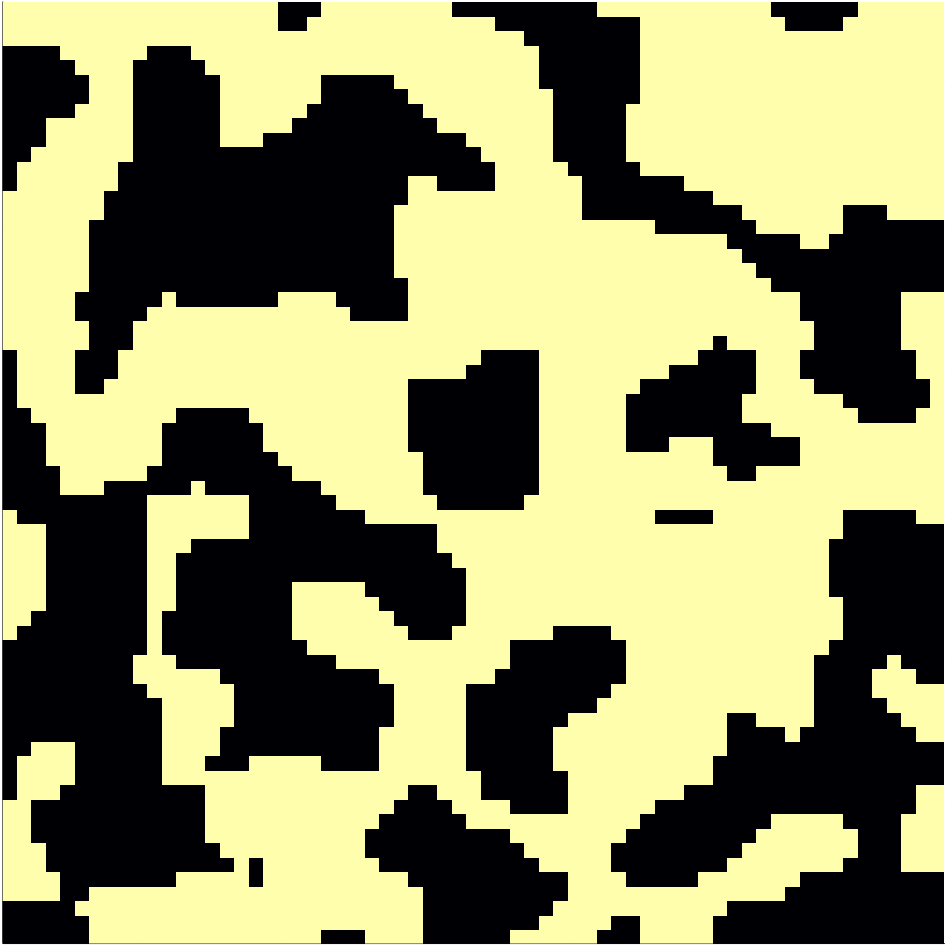}
\includegraphics[width=0.3\textwidth]{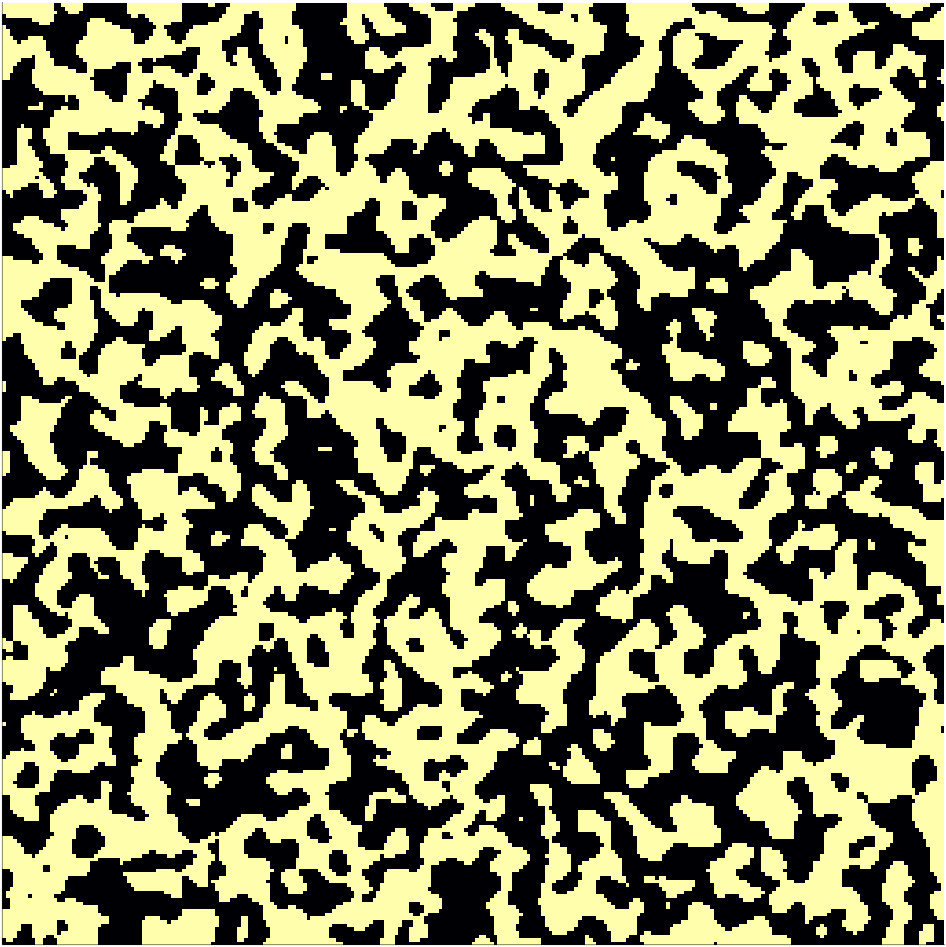}
\caption{A quantized high contrast field with lows of $\rho^{-1}$ and highs of $\rho$ for $n=32$ (left) and $n=128$ (right). 
The features sizes are roughly constant as we increase the mesh size $n$.}
\label{fig:contrast}
\end{figure}

We compare the number of iterations CG \cite{hestenes1952methods} needs to reach a residual of $10^{-12}$. 
In all those experiments, a missing value indicates the factorization was not SPD and, at some point, Cholesky (\autoref{sec:elimination}) failed. 
Given that the problem is defined on a regular mesh, we use a variant of \autoref{algo:ordering_clustering} where the vertex-separator used is based on geometry. This leads to a more regular clustering and, in general, to slightly better performances (in terms of time or memory --- CG iterations and the preconditioner accuracy are usually unaffected).

\autoref{fig:lapl_rho} gives results for $\rho = 1$ to $\rho = 1000$.
We compare the three variants for various $\varepsilon$ and problem size $N = n^2$.
We observe three things from the experiment. First, the number of iterations, particularly at moderate accuracies 
($\varepsilon = 10^{-1}$ or $10^{-2}$), is greatly reduced using \texttt{InS}. 
Further, the \texttt{OrthS} variant is usually the most accurate. 
This is likely due to the improved robustness and accuracies of the orthogonal transformations versus the interpolative 
ones.
Finally, we see that, while the \texttt{In} and \texttt{InS} variants may fail due to non-SPD approximations, the 
\texttt{OrthS} never fails and can always be run, even at $\varepsilon \approx 1$. 
We finally note that the small target residual of $10^{-12}$ in CG illustrates the good numerical properties of the 
preconditioner.
Previous work \cite{ho2016hierarchical} was focused on the interpolative only variant.
Both the scaling and orthogonal transformations greatly improve the algorithm: they reduce the CG 
iteration count and guarantee that the preconditioner stays SPD for SPD problems.

\begin{figure}
\begin{tikzpicture} 
	\begin{groupplot}[
	    group style={
	        group name=my plots,
	        group size=2 by 2,
	        xlabels at=edge bottom,
	        xticklabels at=edge bottom,
	        vertical sep=0.1cm,
			horizontal sep=.1cm,
	        ylabels at=edge left,
	        yticklabels at=edge left,
	    },
		ymax=490,ymin=1,
		xtick={512, 2048, 8192},
		xticklabels={512, 2048, 8192},
		xlabel={$n$},
		ylabel={CG iterations},
		]
        \nextgroupplot[xmode=log,ymode=log,width=5.5cm,height=5cm,ymajorgrids,title style={align=center},ylabel style={align=center},ylabel={CG iterations\\$\rho=1$},title={$\varepsilon = 10^{-2}$}]
			\addplot[mark options={solid},blue,loosely dashed,mark=*       ] table[x=n, y=2FF]{figs/spaND_lapl2_29370319_rho_1.csv};
			\addplot[mark options={solid},hcorange,solid,mark=o            ] table[x=n, y=2TF]{figs/spaND_lapl2_29370319_rho_1.csv};
			\addplot[mark options={solid},red,densely dashed,mark=triangle*] table[x=n, y=2TT]{figs/spaND_lapl2_29370319_rho_1.csv};			
        \nextgroupplot[xmode=log,ymode=log,width=5.5cm,height=5cm,ymajorgrids,legend entries={
			\texttt{In},
            \texttt{InS},
			\texttt{OrthS}
        },legend pos=outer north east,title style={align=center},title={$\varepsilon = 10^{-4}$}]
			\addplot[mark options={solid},blue,loosely dashed,mark=*       ] table[x=n, y=4FF]{figs/spaND_lapl2_29370319_rho_1.csv};
            \addplot[mark options={solid},hcorange,solid,mark=o            ] table[x=n, y=4TF]{figs/spaND_lapl2_29370319_rho_1.csv};
            \addplot[mark options={solid},red,densely dashed,mark=triangle*] table[x=n, y=4TT]{figs/spaND_lapl2_29370319_rho_1.csv};
        \nextgroupplot[xmode=log,ymode=log,width=5.5cm,height=5cm,ymajorgrids,title style={align=center},ylabel style={align=center}, ylabel={CG iterations\\$\rho=100$}]
			\addplot[mark options={solid},blue,loosely dashed,mark=*       ] table[x=n, y=2FF]{figs/spaND_lapl2_29370319_rho_100.csv};
			\addplot[mark options={solid},hcorange,solid,mark=o            ] table[x=n, y=2TF]{figs/spaND_lapl2_29370319_rho_100.csv};
			\addplot[mark options={solid},red,densely dashed,mark=triangle*] table[x=n, y=2TT]{figs/spaND_lapl2_29370319_rho_100.csv};
        \nextgroupplot[xmode=log,ymode=log,width=5.5cm,height=5cm,ymajorgrids,title style={align=center}]
			\addplot[mark options={solid},blue,loosely dashed,mark=*       ] table[x=n, y=4FF]{figs/spaND_lapl2_29370319_rho_100.csv};
			\addplot[mark options={solid},hcorange,solid,mark=o            ] table[x=n, y=4TF]{figs/spaND_lapl2_29370319_rho_100.csv};
			\addplot[mark options={solid},red,densely dashed,mark=triangle*] table[x=n, y=4TT]{figs/spaND_lapl2_29370319_rho_100.csv};	
    \end{groupplot}
\end{tikzpicture}
\caption{2D $n \times n$ Laplacians: each line represent a given variant: \texttt{In} (interpolative and no scaling), 
\texttt{InS} (interpolative and scaling) and \texttt{OrthS} (orthogonal and scaling), at a given accuracy $\varepsilon$. 
Each dot gives the CG iteration count, and we run the experiments on various problem of size
$N = n^2$, for various $\rho$. The conditioning is roughly proportional to $\rho$. A missing data point means Cholesky broke down and
the preconditioner is not SPD.
This shows that, in general, \texttt{InS} and \texttt{OrthS} are much more accurate than \texttt{In} at a given \rev{$\varepsilon$ and \texttt{OrthS} never breaks down}.
In addition, for small enough $\varepsilon$, the accuracy is roughly independent from the problem size $N$. 
Finally, when $\rho$ is not too extreme, there is little dependency with the condition number.
}
\label{fig:lapl_rho}
\end{figure}
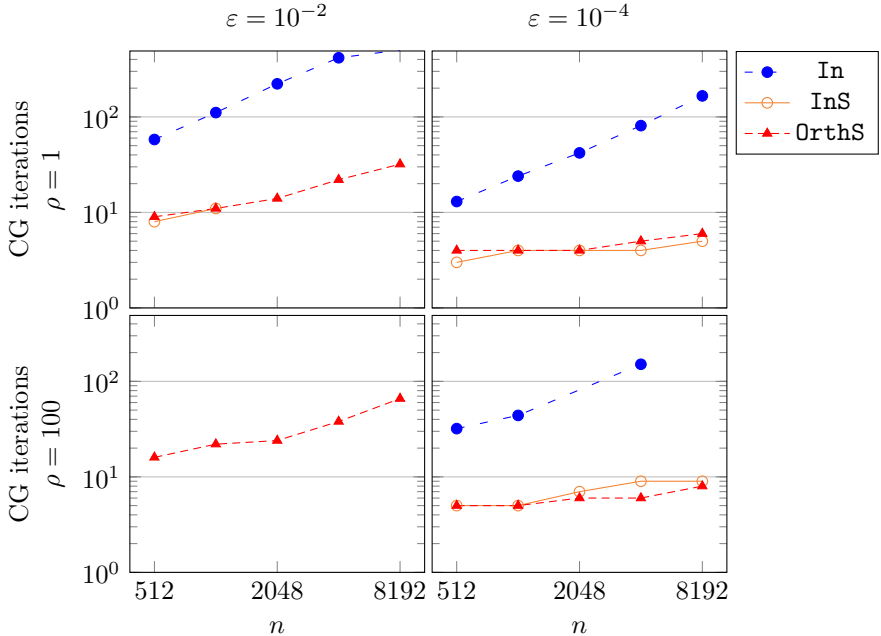

\subsubsection{Non-Regular Problems}

\rev{\autoref{fig:ufl} gives results for three variants on many\footnote{We only excluded \texttt{Queen4147} and \texttt{Bump2911} (for which the solver ran out of memory) as well as the \texttt{Andrews} and \texttt{denormal} cases (which are so ill-conditioned that spaND never converges in less than 500 iterations).} 
of the SPD (real \& square) problems from the 
SuiteSparse matrix collection \cite{davis2011university} with more than $50\,000$ rows and columns.}

Most of these problems come from PDE discretization, but not all. 
\texttt{G2circuit} for instance comes from a circuit simulation problem and \texttt{finan512} comes from a portfolio optimization problem.

For most problems, an accuracy of $\varepsilon = 10^{-2}$ leads to a number of iterations usually less than 100, 
while an accuracy of $\varepsilon = 10^{-4}$ leads usually to less than 10 iterations. 
Only the \texttt{Botonakis/thermomech\_TK} problem needs more than 100 iterations for $\varepsilon = 10^{-6}$.

\begin{figure}
\begin{tikzpicture}
\begin{groupplot}[
    group style={
        group name=my plots,
        group size=1 by 4,
        xlabels at=edge bottom,
        xticklabels at=edge bottom,
        vertical sep=0cm,
        ylabels at=edge left,
        yticklabels at=edge left,
        vertical sep=0.2cm
    },
	xmin=-1,
	xmax=71,
    ymin=1,
    ymax=500,
	footnotesize,
    ybar,
    ytick={1, 3, 10, 30, 100, 300},
    yticklabels={1, 3, 10, 30, 100, 300}, 
    yticklabel style={rotate=-90},
    ylabel style={align=center},
    nodes near coords align={vertical},
    width=0.95\textwidth,
    flexible xticklabels from table={figs/spaND_ufl_all_48519112_tol_0.1.csv}{Matrix}{},
    xticklabel style={tick label style={rotate=90}, font=\tiny},
    xtick=data,
    xtick={0,1,2,3,4,5,6,7,8,9,10,11,12,13,14,15,16,17,18,19,20,21,22,23,24,25,26,27,28,29,30,31,32,33,34,35,36,37,38,39,40,41,42,43,44,45,46,47,48,49,50,51,52,53,54,55,56,57,58,59,60,61,62,63,64,65,66,67,68,69},
	]
\nextgroupplot[
    bar width=0.05cm,
    height=3.5cm,
    legend entries={\texttt{In}, \texttt{InS}, \texttt{OrthS}},
	ymode=log,
	ylabel={\#CG\\$\varepsilon=10^{-1}$},
    ymajorgrids,
]
\addplot[color=black,fill=white]                            table[x expr=\coordindex+0.5, y=FF]{figs/spaND_ufl_all_48519112_tol_0.1.csv};
\addplot[color=black,fill=white,pattern=north east lines, pattern color=hcred, draw=hcred]   table[x expr=\coordindex,     y=TF]{figs/spaND_ufl_all_48519112_tol_0.1.csv};
\addplot[color=black,fill=black]                            table[x expr=\coordindex-0.5, y=TT]{figs/spaND_ufl_all_48519112_tol_0.1.csv};
\nextgroupplot[
    bar width=0.05cm,
    height=3.5cm,
	ymode=log,
    ylabel={\#CG\\$\varepsilon=10^{-2}$},
    ymajorgrids,
]
\addplot[color=black,fill=white]                             table[x expr=\coordindex+0.5, y=FF]{figs/spaND_ufl_all_48519112_tol_0.01.csv};
\addplot[color=black,fill=white,pattern=north east lines, pattern color=hcred, draw=hcred]    table[x expr=\coordindex,     y=TF]{figs/spaND_ufl_all_48519112_tol_0.01.csv};
\addplot[color=black,fill=black]                             table[x expr=\coordindex-0.5, y=TT]{figs/spaND_ufl_all_48519112_tol_0.01.csv};
\nextgroupplot[
    bar width=0.05cm,
    height=3.5cm,
	ymode=log,
    ylabel={\#CG\\$\varepsilon=10^{-4}$},
    ymajorgrids,
]
\addplot[color=black,fill=white]                             table[x expr=\coordindex+0.5, y=FF]{figs/spaND_ufl_all_48519112_tol_0.0001.csv};
\addplot[color=black,fill=white,pattern=north east lines, pattern color=hcred, draw=hcred]    table[x expr=\coordindex,     y=TF]{figs/spaND_ufl_all_48519112_tol_0.0001.csv};
\addplot[color=black,fill=black]                             table[x expr=\coordindex-0.5, y=TT]{figs/spaND_ufl_all_48519112_tol_0.0001.csv};
\nextgroupplot[
    bar width=0.05cm,
    height=3.5cm,
	ymode=log,
    ylabel={\#CG\\$\varepsilon=10^{-6}$},
    ymajorgrids,
]
\addplot[color=black,fill=white]                             table[x expr=\coordindex+0.5, y=FF]{figs/spaND_ufl_all_48519112_tol_1e-06.csv};
\addplot[color=black,fill=white,pattern=north east lines, pattern color=hcred, draw=hcred]    table[x expr=\coordindex,     y=TF]{figs/spaND_ufl_all_48519112_tol_1e-06.csv};
\addplot[color=black,fill=black]                             table[x expr=\coordindex-0.5, y=TT]{figs/spaND_ufl_all_48519112_tol_1e-06.csv};
\end{groupplot}
\end{tikzpicture}
\caption{SuiteSparse matrix collection: result on \rev{many} SPD problems of the SuiteSparse matrix collections with $N \geq 50\,000$. 
The partitioning is only graph-based. Each bar represents the number of CG iterations for a given problem with a given 
variant of the algorithm: \texttt{In} (interpolative and no scaling), 
\texttt{InS} (interpolative and scaling) and \texttt{OrthS} (orthogonal and scaling) at a given accuracy $\varepsilon$. 
The absence of a bar means the algorithm broke down in the face of a non-SPD pivot.
This shows that, as $\varepsilon \to 0$, the algorithm converges on a wide range of problems.
This also shows that the scaling is beneficial in almost all cases. The orthogonal transformations, while not always
better (in terms of accuracy) than the interpolative transformations, do guarantee that the preconditioner stays SPD
and the factorization never breaks down because of indefinite pivots.
}
\label{fig:ufl}
\end{figure}
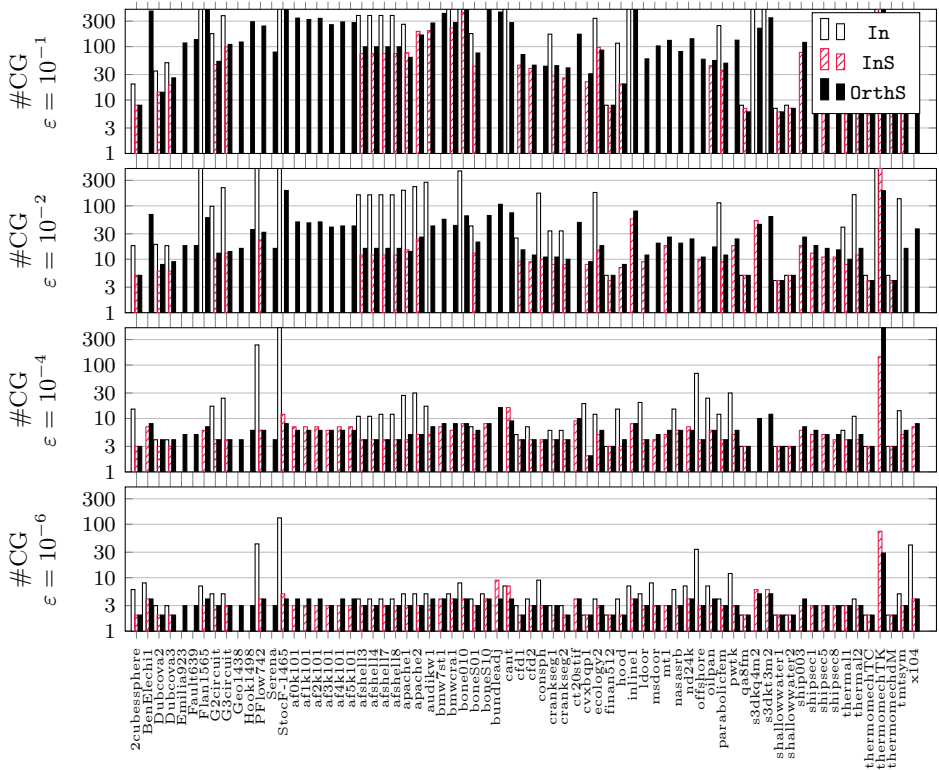

\autoref{fig:profile_ufl} shows a performance profile regarding the CG iteration count. Each plot compares the three
variants for a given accuracy. For a given problem $p$ and a variant $v$, let $CG_{p,v}$ be the CG count and $CG^*_{p}$ 
the best result among the three variants (\rev{\texttt{In}, \texttt{InS} and \texttt{OrthS}}), for a problem $p$. Then each curve is defined as
\[ T_{v}(t) = \frac{\#\left\{p \in P \left| \frac{CG_{p,v}}{CG^*_p} \leq t \right. \right\}}{\#P} \]
Each value $T_{v}(t)$ basically represent the fraction of problems where variant $v$ 
is within $t$ times the best algorithm. \rev{Problems for which the factorization broke down are given $CG_{p,v} = \infty$ and for the others the CG count was capped at 500.}

\rev{
On $\varepsilon=10^{-1}$ and $\varepsilon=10^{-2}$, \texttt{InS} and \texttt{In} often break down, so \texttt{OrthS} is significantly better.
When \texttt{InS} does not break down however, it has similar performances as \texttt{OrthS}.
On $\varepsilon=10^{-4}$, \texttt{InS} rarely breaks down, and performances are very similar to \texttt{OrthS} throughout all the runs.
On $\varepsilon=10^{-6}$, most cases converge in a couples iterations, so the three variants have similar performances.
The plots clearly shows that \texttt{OrthS} is the optimal strategy, being within at most $2$ of the optimal in the worst case,
and being often the winning algorithm. }

\begin{figure}
    \begin{tikzpicture}
        \begin{groupplot}[
            group style={
                group name=my plots,
                group size=2 by 2,
                xlabels at=edge bottom,
                xticklabels at=edge bottom,
                ylabels at=edge left,
                yticklabels at=edge left,
                vertical sep=1cm,
                horizontal sep=0.5cm,
            },
            xtick={1, 2, 5, 10},
            xticklabels={1, 2, 5, 10}, 
            ytick={0, 0.5, 1},
            yticklabels={0, 0.5, 1},
            ymin=-0.1, ymax=1.1,
            xlabel={$t$, performance ratio},
            ylabel={Fraction of\\problems solved},
            ylabel style={align=center},
            ymajorgrids,
            xmajorgrids,
	        ]
            \nextgroupplot[xmode=log,width=6cm,height=5cm,title={$\varepsilon=10^{-1}$}]
                \addplot+[no marks, thick, loosely dashed, black] table[x=t,y=0.1_TT]{figs/spaND_ufl_all_perf_48519112.csv};
                \addplot+[no marks, thick, red] table[x=t,y=0.1_TF]{figs/spaND_ufl_all_perf_48519112.csv};
                \addplot+[no marks, thick, densely dashed, blue] table[x=t,y=0.1_FF]{figs/spaND_ufl_all_perf_48519112.csv};
            \nextgroupplot[xmode=log,width=6cm,height=5cm,title={$\varepsilon=10^{-2}$}]
                \addplot+[no marks, thick, loosely dashed, black] table[x=t,y=0.01_TT]{figs/spaND_ufl_all_perf_48519112.csv};
                \addplot+[no marks, thick, red] table[x=t,y=0.01_TF]{figs/spaND_ufl_all_perf_48519112.csv};
                \addplot+[no marks, thick, densely dashed, blue] table[x=t,y=0.01_FF]{figs/spaND_ufl_all_perf_48519112.csv};
            \nextgroupplot[xmode=log,width=6cm,height=5cm,title={$\varepsilon=10^{-4}$}]
                \addplot+[no marks, thick, loosely dashed, black] table[x=t,y=0.0001_TT]{figs/spaND_ufl_all_perf_48519112.csv};
                \addplot+[no marks, thick, red] table[x=t,y=0.0001_TF]{figs/spaND_ufl_all_perf_48519112.csv};
                \addplot+[no marks, thick, densely dashed, blue] table[x=t,y=0.0001_FF]{figs/spaND_ufl_all_perf_48519112.csv};
                \nextgroupplot[xmode=log,width=6cm,height=5cm,title={$\varepsilon=10^{-6}$},legend pos=south east,legend entries={\texttt{OrthS}, \texttt{InS}, \texttt{In}}]
                \addplot+[no marks, thick, loosely dashed, black] table[x=t,y=1e-06_TT]{figs/spaND_ufl_all_perf_48519112.csv};
                \addplot+[no marks, thick, red] table[x=t,y=1e-06_TF]{figs/spaND_ufl_all_perf_48519112.csv};
                \addplot+[no marks, thick, densely dashed, blue] table[x=t,y=1e-06_FF]{figs/spaND_ufl_all_perf_48519112.csv};
        \end{groupplot}
    \end{tikzpicture}
    \caption{Performance profile for all the SuiteSparse experiments (\autoref{fig:ufl}). Higher is better. 
    The performance criterion is $\#CG$, the number of CG iterations.     
    Each point $T_v(t)$ gives the fraction of problems for which variant $v$ completed with \rev{a CG iterations count} less than $t$ times the best variant. 
    An excellent method is one that starts at $t=1$ close to 1 and quickly reaches 1 as $t$ increases. 
    This means that this method outperforms the other methods in almost all cases.
    \texttt{InS} and \texttt{OrthS} typically have the same number of iterations, but \texttt{InS} sometimes leads to a non-SPD preconditioner, hence the large difference in performances. 
    \texttt{In} typically leads to a much larger iteration count. 
    Looking at the bottom left figure ($\varepsilon = 10^{-4}$) for example, 
    we see that for half of the problems \texttt{In} has a \rev{CG count} more than 10 times greater than the best variant. 
    For all cases, the \texttt{OrthS} is within a factor of 2 of the optimal CG count. 
    This shows the importance of \rev{both the scaling and the orthogonal transformations}.}
    \label{fig:profile_ufl}
\end{figure}
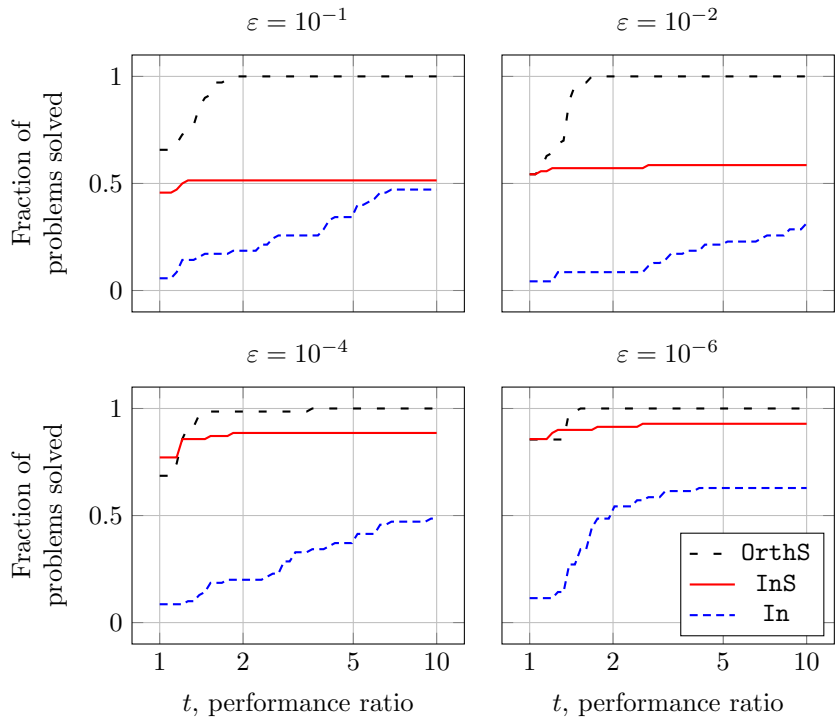

Further, using orthogonal transformations guarantees that the approximation stays SPD, allowing the algorithm to not break down even 
for high $\varepsilon$'s.
The number of iterations of \texttt{OrthS} is not always strictly smaller than the \texttt{InS} variant, while it is for the regular Laplacian examples. 
However, the extra robustness (no need for pivoting) of the orthogonal transformations make them quite attractive in 
practice for SPD problems.

We also point out that previous work \cite{ho2016hierarchical} was restricted to standard elliptic model problems. 
To the best of our knowledge, this is the first application of this algorithm to a wide range of problems.

\subsection{Scalings with problem size}\label{sec:scalings_with_problem_size}

We now consider scalings, i.e., how does the algorithm perform as $N$ grows.
\autoref{fig:scaling_laplacians} shows the evolution of the top separator size right before elimination (top) and the 
number of CG iterations (bottom) for $\rho = 1$ and $\rho = 100$ for \emph{3D} problems generated as in \autoref{sec:2D}
with a classic 7-points stencil.
\rev{From} now on, we will only consider the scaling \& orthogonal method (\texttt{OrthS}).

This figure shows two properties of the algorithm:
\begin{itemize}
    \item the top separator size ($\st$) typically grows like $\OO{N^{1/3}}$, regardless of $\varepsilon$;
    \item for small enough $\varepsilon$, the number of CG iterations is roughly $\OO{1}$.
\end{itemize}
Combining those two properties, we can expect \rev{(see \autoref{sec:complexity})}, for small enough $\varepsilon$,
\begin{itemize}
    \item a factorization time of \rev{$\OO{N \log N}$};
    \item a solve time of \rev{$\OO{N \cdot 1} = \OO{N}$},
\end{itemize}
which implies that the algorithm scales roughly linearly with $N$.

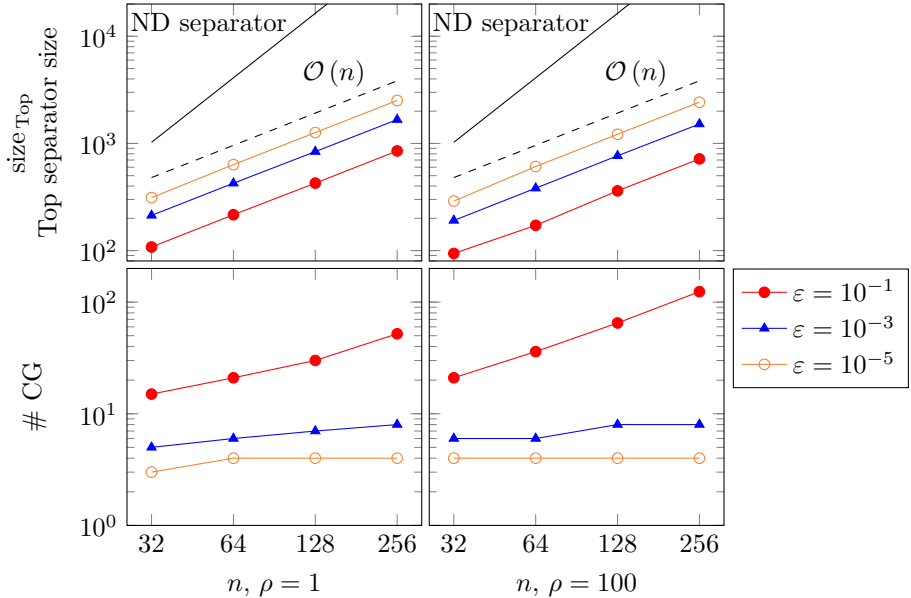
\begin{figure}
    \centering
    \begin{tikzpicture}
        \begin{groupplot}[
            group style={
                group name=lapl3d_1,
                group size=2 by 2,
                xlabels at=edge bottom,
                xticklabels at=edge bottom,
                vertical sep=0.1cm,
                horizontal sep=0.1cm,
                ylabels at=edge left,
                yticklabels at=edge left,
            },
            xtick={32,64,128,256},
            xticklabels={32,64,128,256},
        ]
        \nextgroupplot[xmode=log,ymode=log,width=5.5cm,height=5cm,ylabel style={align=center},ylabel={$\st$\\Top separator size},ymin=80,ymax=20000]
            \addplot [red,mark=*]          table[x=N, y=stop]{figs/lapl3d_1_1e-1.dat};
            \addplot [blue,mark=triangle*] table[x=N, y=stop]{figs/lapl3d_1_1e-3.dat};
            \addplot [hcorange,mark=o]     table[x=N, y=stop]{figs/lapl3d_1_1e-5.dat};
            \addplot [domain=32:256,solid]{x*x};
            \addplot [domain=32:256,dashed]{15*x};
            \node at (axis cs:52,13000) {ND separator};
            \node at (axis cs:150,4500) {$\OO{n}$};
        \nextgroupplot[xmode=log,ymode=log,width=5.5cm,height=5cm,ymin=80,ymax=20000]
            \addplot [red,mark=*]          table[x=N, y=stop]{figs/lapl3d_100_1e-1.dat};
            \addplot [blue,mark=triangle*] table[x=N, y=stop]{figs/lapl3d_100_1e-3.dat};
            \addplot [hcorange,mark=o]     table[x=N, y=stop]{figs/lapl3d_100_1e-5.dat};
            \addplot [domain=32:256,solid]{x*x};
            \addplot [domain=32:256,dashed]{15*x};
            \node at (axis cs:52,13000) {ND separator};
            \node at (axis cs:150,4500) {$\OO{n}$};
        \nextgroupplot[xmode=log,ymode=log,width=5.5cm,height=5cm,ylabel={\# CG},ymin=1,ymax=200,xlabel={$n$, $\rho=1$}]
            \addplot [red,mark=*]          table[x=N, y=CG]{figs/lapl3d_1_1e-1.dat};
            \addplot [blue,mark=triangle*] table[x=N, y=CG]{figs/lapl3d_1_1e-3.dat};
            \addplot [hcorange,mark=o]     table[x=N, y=CG]{figs/lapl3d_1_1e-5.dat};
        \nextgroupplot[xmode=log,ymode=log,width=5.5cm,height=5cm,ymin=1,ymax=200,xlabel={$n$, $\rho=100$},legend entries={$\varepsilon=10^{-1}$, $\varepsilon=10^{-3}$, $\varepsilon=10^{-5}$},legend pos=outer north east]
            \addplot [red,mark=*]          table[x=N, y=CG]{figs/lapl3d_100_1e-1.dat};
            \addplot [blue,mark=triangle*] table[x=N, y=CG]{figs/lapl3d_100_1e-3.dat};
            \addplot [hcorange,mark=o]     table[x=N, y=CG]{figs/lapl3d_100_1e-5.dat};
        \end{groupplot}
    \end{tikzpicture}
    \caption{3D $n \times n \times n$ Laplacian results for $\rho = 1$ (left) and $\rho = 100$ (right) using \texttt{OrthS}. We see that $\st$ (top separator final
size, i.e., right before elimination)  scales like $\OO{n}$, and that increasing the accuracy (decreasing $\varepsilon$) 
essentially adds a constant; it does not change the scaling. This should be compared with the classical ND separator size
(solid like), equal to $n^2$.
In addition, for a small enough $\varepsilon$, the CG iteration count becomes virtually constant.
Both those \rev{facts} mean the algorithm can be expected to have complexity $\OO{N}$ (see \autoref{sec:complexity}).}
    \label{fig:scaling_laplacians}
\end{figure}

\subsection{Timings and Memory Usage}

We now study the efficiency of the algorithm in terms of time (factorization and solve time) and memory usage on ``real-life'' problems.
To evaluate our algorithm, we use the following two metrics:
\begin{itemize}
    \item the factorization and solve time ($\tf$ and $\ts$);
    \item the memory footprint ($\mf$, the number of non-zeros in the preconditioner $M$).
\end{itemize}

\paragraph{SuiteSparse} \autoref{tab:ufl_results} shows the results on two specific problems from the SuiteSparse collection \cite{davis2011university}, 
inline and audikw. 
For both problems, we see that the ``sweet-spot'' in terms of minimal time-to-solution is not for high $\varepsilon$, but for much smaller $\varepsilon$. 
For the audikw problem, the optimal is when using $\varepsilon = 10^{-2}$, and for inline, $10^{-4}$ gives optimal result. The $\st$ for inline are overall much smaller than for audikw. This is usually an indication that the problem is near 2D, for which $\st$ is typically $\OO{1}$.
Those problems are of fairly small size and, as such, direct solvers (with smaller constants and better implementations) 
remain competitive.

\begin{table}
    \centering
    \begin{tabular}{l|lllllll}
        Problem ($N$)   & $\varepsilon$         & $\tp$ (s.) & $\tf$ (s.)    & $\ts$ (s.)    & $\ncg$      & $\st$                 & $\mf$ $(10^9)$    \\
        \hline
        audikw\_1       & $10^{-1}$             & \rev{96}   & \rev{128}     & \rev{512}     & \rev{277}   & \rev{322}             & \rev{0.46}        \\ 
        943\,695        & $10^{-2}$             & \rev{95}   & \rev{268}     & \rev{103}     & \rev{42}    & 606                   & \rev{0.73}        \\
                        & $10^{-4}$             & \rev{95}   & \rev{500}     & \rev{18}      & \rev{7}     & 1175                  & \rev{1.08}        \\
        \hline
        inline\_1       & $10^{-1}$             & \rev{40}   & \rev{8}       & \rev{$>224$}  & $>500$      & \rev{1}               & \rev{0.11}        \\
        503\,712        & $10^{-2}$             & \rev{41}   & \rev{13}      & \rev{44}      & \rev{80}    & 13                    & \rev{0.13}        \\
                        & $10^{-4}$             & \rev{41}   & \rev{18}      & \rev{5}       & \rev{8}     & 19                    & \rev{0.16}        \\
    \end{tabular}
    \caption{Some SuiteSparse performance results using \texttt{OrthS}. Completely general partitioning (no geometry information used) 
             \rev{using \autoref{algo:ordering_clustering} with Metis as a vertex-separation routine}.
We see that the algorithm \rev{does} converge when $\varepsilon \to 0$. The sweet spot varies for both problems.
Notice that inline\_1 has a very small $\st$, characteristic of near-2D problems, \rev{while} the top separator has a much
larger size for audikw\_1.}
    \label{tab:ufl_results}
\end{table}

\paragraph{Ice-sheet modeling problem}
\autoref{tab:ice_sheet_results} gives the result on an ice-sheet modeling problem \cite{tezaur2015albany}. 
This problem comes from the modeling of ice flows on Antarctica using a finite-element discretization. 
The problem is challenging because of the high variations in background field and the near-singular blocks in the matrix, leading to a condition number of more than $10^{11}$.  This problem is nearly 2D. The graph in the $x, y$ plane is regular but non-square. 
It is then extruded in the $z$ direction.

We illustrate the partitioning (top-left) and one layer of the solution (top-right, with a random right-hand side) on a log scale. Note the high variations in scales in the solution. This makes the problem very ill conditioned and hard to solve with classical preconditioners.
Since the problem is (nearly) 2D and we are given the geometry, we partition the matrix in the $xy$ plane and extrude the
partitioning in the $z$ direction. The partitioning uses a classical recursive coordinate bisection approach \cite{berger1987partitioning}.

We use two sequences of matrices with a different number of layers in the $z$ direction. 
We see that $\st$ grows very slowly, close (but not exactly) like $\OO{1}$ for each set of problems. 
This is typical of 2D or near-2D problems. 
The memory use is roughly linear for each set of problems, and the factorization time is growing almost linearly. 
This validates the effectiveness of the algorithm.

We also compared the algorithm against a direct method (simply using \algo{} with no compression but otherwise with the same parameters). 
The results are in the ``Direct'' column. 
We note the very poor scaling of the direct method; our algorithm, on the other hand, performs much better.
In addition, we also compared the algorithm to out-of-the-box algebraic multigrid (AMG, a classical AMG) and Incomplete LU(0). 
On this specific problem, AMG simply did not converge in less than 500 iterations, the residual stalling around 1.0. 
While specifically designed AMG can and does solve this problem well \cite{tezaur2015albany}, 
this illustrates that out-of-box algorithms cannot always efficiently solve very ill-conditioned problems. 
Because of this, we do not report those results.
We finally tested Ifpack2's ILU(0) \cite{Ifpack2} with GMRES. We tested two \rev{orderings}, horizontal (layer-wise) and vertical (column-wise). 
The layer-wise ordering gave (by far) the best performances and we report only this one.
However, while it is competitive for small problems, it cannot solve large problems because the number of iterations grows quickly, 
making the algorithm too expensive. 
This illustrate the strong advantage of spaND: with a nearly constant number of iterations, we do not suffer from this deterioration of the preconditioner 
and can solve larger problems.

We note that those results can also be compared with recent work using LoRaSp \cite{pouransari2017fast, chena2018robust} on the same matrices. 
Overall, while the scaling with $N$ is similar, spaND exhibits better constants.

\begin{table}
    \centering
    \includegraphics[width=0.45\textwidth]{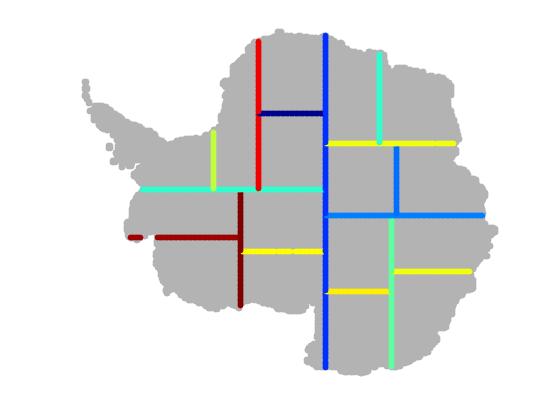}
    \includegraphics[width=0.45\textwidth]{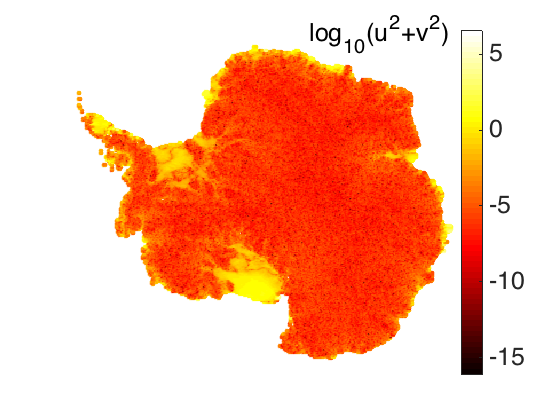}\\
    {\small
     \setlength{\tabcolsep}{3pt}
    \begin{tabular}{l|llllll|l|l}
                             & \multicolumn{3}{l}{spaND}                     &               &               &                       & Direct            & ILU(0)          \\
        $N$                  & $\tp$         & $\tf$         & $\ts$         & $\ncg$        & $\st$         & $\mf$                 & $\tf+\ts$         & $\ts$ ($\text{n}_{\text{GMRES}}$)  \\
                             & (s.)          & (s.)          & (s.)          &               &               & $(10^9)$              & (s.)              & (s.)           \\
        \hline
        5 layers             & & & & & & & & \\
        % 5 layers
        629\,544 (16 km)     & \rev{2}       & \rev{7  }     & \rev{3 }      & 7             & 78            & 0.15                  & \rev{19}          & \rev{23} (92)          \\
        2\,521\,872 (8 km)   & \rev{10}      & \rev{28 }     & \rev{14}      & 8             & 88            & 0.59                  & \rev{126}         & \rev{286} (182)        \\
        10\,096\,080 (4 km)  & \rev{50}      & \rev{124}     & \rev{89}      & 10            & 99            & 2.40                  & \rev{1036}        & \rev{7137} (720)       \\
        \hline
        10 layers            & & & & & & & & \\
        % 10 layers
        1\,154\,164 (16 km)  & \rev{4}       & \rev{23 }     & \rev{7  }     & 7             & 137           & 0.42                  & \rev{86}          & \rev{42} (93)          \\
        4\,623\,432 (8 km)   & \rev{20}      & \rev{97 }     & \rev{34 }     & 8             & 147           & 1.73                  & \rev{725}         & \rev{544} (181)        \\
        18\,509\,480 (4 km)  & \rev{100}     & \rev{538}     & \rev{311}     & 10            & 159           & 6.80                  & ---               & \rev{18680} (745)      \\
    \end{tabular}}
    \\ \vspace{0.2cm}
    \begin{tikzpicture}
        \begin{loglogaxis} [
                width=2in,
                height=2in,
                xtick={6, 24, 96},
                xticklabels={16km, 8km, 4km},
                ylabel={$\tf + \ts$ (s.)},
            ]
            \addplot [domain=6:96,dashed,black]{15*x};
            \addplot [red,mark=*] coordinates { (6,30) (24,131) (96,849) };
            \addplot [blue,mark=triangle*] coordinates { (6,10) (24,42) (96,213) };
        \end{loglogaxis}
    \end{tikzpicture}
    \begin{tikzpicture}
        \begin{loglogaxis} [
                legend pos=outer north east,
                width=2in,
                height=2in,
                xtick={6, 24, 96},
                xticklabels={16km, 8km, 4km},
                ylabel={$\mf$ ($10^9$)},
            ]
            \addplot [domain=6:96,dashed,black]{1e-1*x};
            \addplot [red,mark=*] coordinates { (6,0.42) (24, 1.73) (96,6.80) };
            \addplot [blue,mark=triangle*] coordinates { (6,0.15) (24, 0.59) (96,2.40) };
            \legend {$\OO{N}$,10 layers, 5 layers};
        \end{loglogaxis}
    \end{tikzpicture}
    \caption{Ice Sheet results. Unregular geometric partitioning, $\varepsilon = 10^{-2}$, \texttt{OrthS}. The top left picture illustrates the separators (for 
the top 5 levels) and the top right picture shows the solution (for a random right-hand side $b$) on log scale. 
--- indicates the direct method ran out of memory.
We run ILU with 2 ordering: layer-wise ordering and vertical column-wise ordering. The later lead to very poor convergence and is not shown here.
This problem is very ill-conditioned and typically very hard to solve using out-of-the-box preconditioners. spaND, on the
other hand, solves the problem well and \rev{scales} near-linearly with the problem size.}
    \label{tab:ice_sheet_results}
\end{table}

\paragraph{SPE benchmark}
\autoref{fig:spe_results} gives the results on a cubic slide of the SPE (Society of Petroleum Engineering) benchmark \cite{christie2001tenth}, a classical benchmark to evaluate oil \& gas exploration codes. 
This matrix models a porous media flow. 
This problem is particularly challenging for direct methods since it resembles a 3D cubic problem and \rev{leads} to a high complexity. 
On the other hand, it can be solved quite efficiently with classical preconditioners like AMG or ILU.

We use various values of $n$, and the problem is then of size $N = n^3$. 
The bottom pictures show the $\st$ and $\mf$ scaling with $N$. 
We see that $\st$ grows roughly like $\OO{N^{1/3}}$; this is typical of $3D$ problems. 
The memory use grows linearly with $N$. Furthermore, the number of CG iterations is constant for all resolutions. 
This serves \rev{as} another validation of the ability of spaND to solve large-scale problems. 
In the last column, we show the result using the direct solver. 
Since it is a direct solver, the memory use is too \rev{great}, and we cannot solve the \rev{8M} problem. 
Furthermore, the time to solve the 2M problem is about 10 times more than using spaND. 
This shows the limitations of direct \rev{solvers} for solving large 3D problems for which the fill-in is too significant.

\begin{table}
    \centering
    \includegraphics[width=0.45\textwidth]{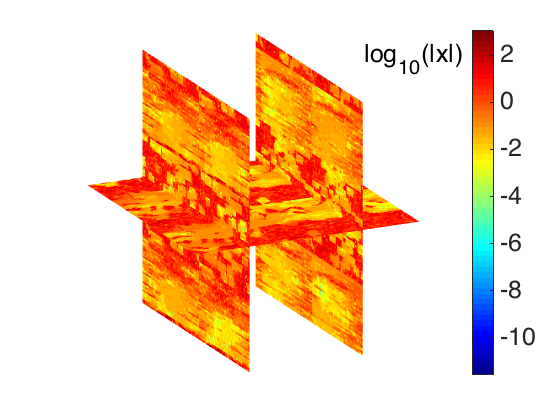}\\
    \begin{tabular}{ll|llllll|l}
                    &                 & \algo{}       &               &               &               &               &                       & Direct.     \\
        $n$         & $N = n^3$       & $\tp$         & $\tf$         & $\ts$         & $\ncg$        & $\st$         & $\mf$                 & $t_{F}+\ts$  \\
                    &                 & (s.)          & (s.)          & (s.)          &               &               & ($10^9$)              & (s.) \\   
        \hline
        $128$       & $2\,097\,152$   & \rev{7}       & \rev{55}      & \rev{18}      & 13            & 504           &  0.62                 & \rev{743}     \\
        $160$       & $4\,096\,000$   & \rev{18}      & \rev{118}     & \rev{44}      & 14            & 635           &  1.2                  & \rev{3677}    \\
        $200$       & $8\,000\,000$   & \rev{40}      & \rev{254}     & \rev{102}     & 16            & 962           &  2.5                  & ---      \\
        $252$       & $16\,003\,008$  & \rev{87}      & \rev{650}     & \rev{256}     & 14            & 891           &  5.0                  & ---      \\ 
    \end{tabular} \\ \vspace{0.2cm}
    \begin{tikzpicture}
        \begin{loglogaxis} [
                legend pos=north west,
                width=2in,
                height=2in,
                xtick={2097152,4096000,8000000,16003008},
                xticklabels={2M, 4M, 8M, 16M},
                ylabel={$\st$},
            ]
            \addplot [domain=2000000:16000000,dashed,red]{0.04*x^(2/3)};
            \addplot [domain=2000000:16000000,dashed,black]{3*x^(1/3)};
            \addplot [blue,mark=*] coordinates { (2097152, 504) (4096000,635) (8000000,962) (16003008,891) };
            \node at (axis cs:4000000,1700) {$\OO{N^{2/3}}$};
            \node at (axis cs:11000000,500) {$\OO{N^{1/3}}$};
        \end{loglogaxis}
    \end{tikzpicture}
    \begin{tikzpicture}
        \begin{loglogaxis} [
                legend pos=north west,
                width=2in,
                height=2in,
                xtick={2097152,4096000,8000000,16003008},
                xticklabels={2M, 4M, 8M, 16M},
                ylabel={$\mf$},
            ]
            \addplot [domain=2000000:16000000,dashed,red]{1e9*0.000000005*x^(4/3)};
            \addplot [domain=2000000:16000000,dashed,black]{1e9*0.0000005*x};
            \addplot [blue,mark=*] coordinates { (2097152, 0.65e9) (4096000,1.2e9) (8000000,2.5e9) (16003008,5.0e9) };
            \node at (axis cs:4500000,1e10) {$\OO{N^{4/3}}$};
            \node at (axis cs:12000000,8e9) {$\OO{N}$};
        \end{loglogaxis}
    \end{tikzpicture}
    \caption{SPE results. Regular geometric partitioning, $\varepsilon = 10^{-2}$, \texttt{OrthS}. ``---'' indicates the direct method ran out of memory.
This problem is a regular cube, and hence very hard to solve using a direct method, since the separators are very large.
 spaND does not suffer from this problem and can solve this problem well, with a near (but not exactly) linear scaling
with the problem size.}
    \label{fig:spe_results}
\end{table}

\subsection{Profiling}

\autoref{fig:profiling_mem} shows the (cumulative) memory taken by $M$ in \algo{}, compared to the direct method.
This shows clearly the effect of the approximation. 
At the beginning, memory increases slowly. 
Then, we keep eliminating and going up the tree and elimination becomes more and more expansive. 
The sparsification, however, allows us to greatly decrease the memory use by reducing the separator's sizes. 
In this specific example, sparsification is skipped for the first four levels. 
This can be seen on \autoref{fig:profiling_mem}, where \algo{}'s level 5 memory use is slightly greater than the Direct method. 
After that, however, it remains almost constant.
\begin{figure}
    \centering
    \begin{tikzpicture}
        \begin{axis}[
            legend pos=north west,
            width=4in,
            height=2in,
            xlabel={Level},
            ylabel={Non zeros in $M$},
            legend entries={Direct, \algo{}},
            ]
        \addplot[red,mark=*] table[x=lvl,y=direct]{figs/mem_spe_4m.dat};
        \addplot[blue,mark=triangle*] table[x=lvl,y=spand]{figs/mem_spe_4m.dat};
        \end{axis}
    \end{tikzpicture}
    \caption{Memory profiling of the SPE 4M problem. Each dot shows the total (cumulative) memory used by the partial 
    preconditioner up to this level in the elimination. We compare spaND to a direct method using Nested Dissection.
Thanks to the sparsification (started at level \rev{5}), the memory stays well under control, while a direct method takes more
and more memory as the elimination proceeds.}
    \label{fig:profiling_mem}
\end{figure}
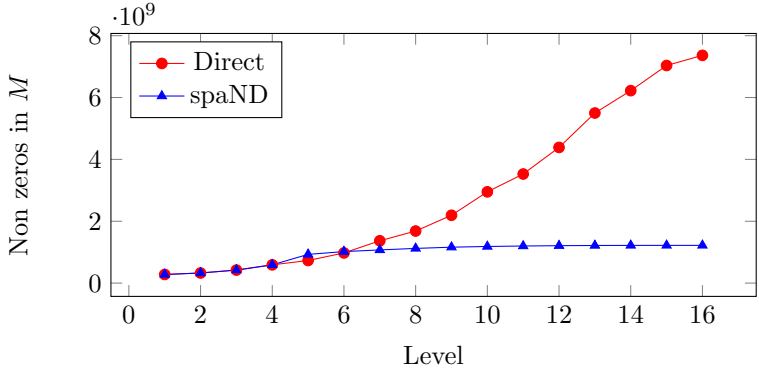

\autoref{fig:profiling} shows profiling (traces) when solving a larger 16M SPE problem. This clearly shows the advantage 
of the algorithm. When using a direct method, elimination becomes excessively slow when reaching the top of the tree, 
and the time spent at the last level usually dominates. For instance in this specific problem, the last elimination would 
require factoring a matrix of size approximately $252^2 \times 252^2 = 63\,504 \times 63\,504$ (approx. 32GB!) that is completely 
dense. 
Our algorithm, on the other hand, spends more time at the early levels in the tree eliminating dofs and sparsifying
separators (see the large brown bar at level 5). 
As a result, the time actually \emph{decreases} as we reach higher levels in the tree. 
This makes for a much more efficient solver.

Notice that in this example, we start the sparsification at level 5 (i.e., we skip it for four levels). 
In our experiments, this gives the best results. Starting earlier leads to very high ranks (i.e., there is not
much to compress), while delaying it too much leads to too large matrices $A_{pn}$ for which RRQR becomes excessively
slow.

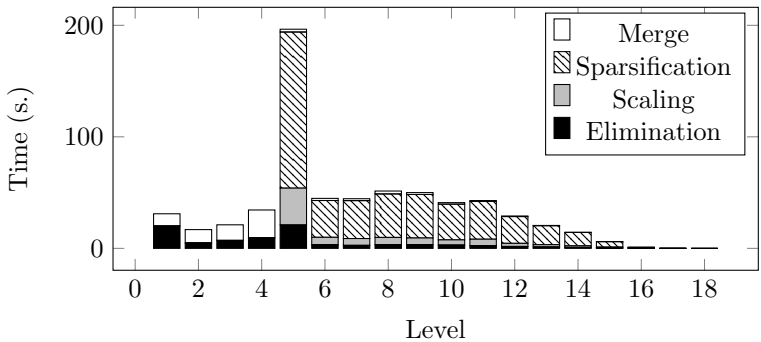
\begin{figure}
    \centering
    \begin{tikzpicture}
        \begin{axis}[
            ybar stacked,
            bar width=10pt,
            width=4in,
            height=2in,
            legend entries={Elimination, Scaling, Sparsification, Merge},
            reverse legend,
            xlabel={Level},
            ylabel={Time (s.)},
            ]
        \addplot+[ybar,color=black,fill=black]                          table[x=lvl,y=elim]{figs/time_spe_16m.dat};
        \addplot+[ybar,color=black,fill=black!25]                         table[x=lvl,y=scal]{figs/time_spe_16m.dat};
        \addplot+[ybar,color=black,fill=white,pattern=north west lines] table[x=lvl,y=spar]{figs/time_spe_16m.dat};
        \addplot+[ybar,color=black,fill=white]                          table[x=lvl,y=merg]{figs/time_spe_16m.dat};
        \end{axis}
    \end{tikzpicture}
    \caption{Time profiling of the SPE 16M problem. Each bar represents the time spent at each level in the elimination.
Unlike direct methods, most of the compute time is spent at the first levels (near the leaves), where we have to solve
many small problems. A direct method would likely be faster at the beginning, but much slower near the end, where the
fronts become very large and have to be factored exactly. \rev{Sparsification time spikes at level 5 when it is triggered. 
Starting sparsification sooner is inefficient since the blocks are not low-rank enough, and the time spent in the low-rank factorizations is then wasted.}}
    \label{fig:profiling}
\end{figure}

\section{Conclusion}

In this paper we developed a sparsified Nested Dissection algorithm. 
The algorithm combines ideas from Nested Dissection (a fast direct method) and low-rank approximations to 
reduce the separator sizes. The result is an approximate factorization that can be computed in near-linear time and
results in an efficient preconditioner.

We note that it differs from the ``classical'' way of accelerating sparse direct solvers (like MUMPS with BLR and Pastix
with HODLR). 
Instead of using $\H$-algebra to compress large fronts, it simply keeps the fronts small throughout the algorithm by 
sparsifying them at each step of the algorithm. 

Prior work in this area included the HIF algorithm \cite{ho2016hierarchical}. 
While our work resembles it, HIF is limited to $n \times n \times n$ regular problems \cite{ho2016hierarchical} and
does not use either the block diagonal scaling or orthogonal transformations.
The LoRaSp algorithm \cite{pouransari2017fast} is also similar.
LoRaSp's performances however may degrade when the ranks at the leaf level are not small and does not have the same sparsity guarantees \cite{chena2018robust}. 
The ordering and the ability to skip compression for some levels fixes this.

We discuss three variants of the algorithm, depending on the low-rank approximation methods (interpolative or orthogonal)
and the prior use, or not, of scaling.
We showed through extensive numerical experiments that the scaling has a large impact on the preconditioner's accuracy.
In addition, the use of orthogonal transformation implies that the algorithm does not break down even when $\varepsilon \approx 1$.

We then tested the algorithm on both ill-conditioned problems (typically hard for preconditioners) and ``cubic'' problems
(typically hard for direct methods). On these problems, spaND is very efficient, with very favorable
scaling for the factorization and near-constant CG iteration count.

Multiple research directions remain unexplored.
The compression algorithm used was a simple (but still quite expensive) RRQR algorithm. Other fast algorithms could be
used, like randomized methods or skeletonized interpolation (where the $\ske$ of the interpolative factorization are picked
a-priori using some heuristic). These techniques could greatly accelerate the compression step.
The loss of accuracy remains to be studied.

Expanding the algorithm to non-SPD or non-symmetric systems is conceptually straightforward. 
If $A$ is not SPD, one can simply use the $LDL^\top$ factorization in the elimination step. Note that in this case, 
(symmetric) pivoting may be required and the algorithm may break down.
If $A$ is not symmetric, we need to compress $A_{pn}$ and $A_{np}$
and use the obtained basis on both the left and the right.
The resulting preconditioner can then be coupled with GMRES instead of CG.

The partitioning algorithm is well-suited for matrices arising from the discretization of elliptic PDE's, 
where we know that well-separated clusters have low-numerical rank. It would be interesting to explore other partitioning
algorithms, for instance for indefinite matrices coming from Maxwell's equations.

Finally, we mention that spaND exhibits more parallelism than direct methods. 
Indeed, most of the work occurs near the leaves of the tree. This means less synchronization and more parallelism. 
This is in contrast with direct methods based on Nested Dissection where the bottleneck is usually the factorization of 
the top separator at the root of the tree.

\section*{Acknowledgement}

Some of the computing for this project was performed on the Sherlock cluster. We would like to thank Stanford University and the Stanford Research Computing Center for providing computational resources and support that contributed to these research results.

This work was partly funded by the U.S. Department of Energy National Nuclear Security Administration under Award Number DE-NA0002373-1 
and partly funded by an LDRD research grant from Sandia National Laboratories. 
Sandia National Laboratories is a multimission laboratory managed and operated by National Technology and Engineering Solutions of Sandia, 
LLC., a wholly owned subsidiary of Honeywell International, Inc., for the U.S. Department of Energy's National Nuclear Security Administration under contract DE-NA-0003525.

We thank Cindy Orozco for the idea of the modified Nested Dissection algorithm, and Bazyli Klockiewicz for his insightful comments on the numerical algorithms.

\bibliography{references}
\bibliographystyle{siamplain}

\end{document}